\documentclass[12pt]{amsart}
\usepackage{amscd, amssymb, pb-diagram}
\usepackage{tabto}
\usepackage{amsthm}

\usepackage[margin=2.1cm, includehead, includefoot]{geometry}

\def\max{\operatorname{max}}
\def\min{\operatorname{min}}

\def\Mod{\operatorname{mod}}
\def\lcm{\operatorname{lcm}}
\def\mult{\operatorname{mult}}
\def\add{\operatorname{add}}

\def\F{\mathbb{F}}
\def\R{\mathbb{R}}
\def\N{\mathbb{N}}
\def\Z{\mathbb{Z}}

\def\Q{\mathbb{Q}}

\def\AA{\mathcal{A}}

\def\JJ{\mathcal{J}}

\def\OO{\mathcal{O}}

\def\TT{\mathcal{T}}

\newcommand{\inv}{^{-1}}
\newcommand{\Jj}{\mathcal{J}}

\newcommand{\NN}{\mathbb{N}}

\newcommand{\QQ}{\mathcal{Q}}

\newcommand{\nxn}{\N \rtimes \N^\times}
\newcommand{\zxz}{\Z \rtimes \Z^\times}
\newcommand{\qxq}{{\mathbb Q \rtimes \mathbb Q^*_+}}

\newcommand{\ZS}{Zappa-Sz\'{e}p }
\renewcommand{\mid}{:}

\def\BS{\operatorname{BS}}

\def\zsp{Zappa-Sz\'{e}p product }
\def\zsps{Zappa-Sz\'{e}p products}

\makeatletter
\def\imod#1{\allowbreak\mkern10mu({\operator@font mod}\,\,#1)}
\makeatother

\newtheorem{thm}{Theorem}[section]
\newtheorem{cor}[thm]{Corollary}

\newtheorem{lemma}[thm]{Lemma}
\newtheorem{prop}[thm]{Proposition}

\theoremstyle{definition}
\newtheorem{definition}[thm]{Definition}

\theoremstyle{remark}
\newtheorem{remark}[thm]{Remark}

\newtheorem{example}[thm]{Example}

\begin{document}

\title[Zappa-Sz\'{e}p products of semigroups and their $C^*$-algebras]{Zappa-Sz\'{e}p products of semigroups and their $C^*$-algebras}

\author{Nathan Brownlowe}
\address{Nathan Brownlowe, Jacqui Ramagge, David Robertson and Michael F. Whittaker \\ School of Mathematics and
Applied Statistics  \\
The University of Wollongong\\
NSW  2522\\
AUSTRALIA} \email{nathanb@uow.edu.au,ramagge@uow.edu.au,droberts@uow.edu.au,mfwhittaker@gmail.com}
\author[Jacqui Ramagge]{Jacqui Ramagge}
\author[David Robertson]{David Robertson}
\author[Michael F. Whittaker]{Michael F. Whittaker}

\thanks{This research was supported by the Australian Research Council}


\begin{abstract}
{\zsps} of semigroups provide a rich class of examples of semigroups that include the self-similar group actions of Nekrashevych. We use Li's construction of semigroup $C^*$-algebras to associate a $C^*$-algebra to {\zsps} and give an explicit presentation of the algebra.  We then define a quotient $C^*$-algebra that generalises the Cuntz-Pimsner algebras for self-similar actions. We indicate how known examples, previously viewed as distinct classes, fit into our unifying framework. We  specifically discuss the Baumslag-Solitar groups, the binary adding machine, the semigroup $\nxn$, and the $ax+b$-semigroup $\zxz$.
\end{abstract}

\maketitle

\section{Introduction}\label{sec: introduction}

Examples are crucial to progress in $C^*$-algebras. Operator-algebraists are therefore enthusiastic to have ways of generating and analysing rich classes of examples. Semigroups feature in a number of families of interesting examples. In this article we describe a new class of semigroup $C^*$-algebras.

The theory of $C^*$-algebras associated to semigroups can be traced back to Coburn's Theorem~\cite{Cob}, which says that any two $C^*$-algebras generated by a non-unitary isometry are isomorphic. There have been a number of generalisations of Coburn's Theorem, including Douglas's work~\cite{d} on positive cones of ordered subgroups of~$\R$, and Murphy's work~\cite{m} on positive cones in ordered abelian groups. A major generalisation was developed by Nica~\cite{n} through his introduction of quasi-lattice ordered groups $(G,P)$.

A quasi-lattice ordered group $(G,P)$ consists of a  partially-ordered group $G$ and a positive cone $P$ in $G$. Nica identified a class of covariant isometric representations of $P$, and introduced the $C^*$-algebra $C^*(G,P)$ universal for such representations. Quasi-lattice ordered groups are rigid enough in their structure to produce a tractable class of $C^*$-algebras $C^*(G,P)$, and yet they include a wide range of interesting semigroups as examples. Indeed, quasi-lattice ordered groups are still providing a rich source of interesting $C^*$-algebras, as evidenced by the recent work on the $C^*$-algebras associated to $\nxn$ \cite{lr}, and the Baumslag-Solitar groups~\cite{Sp}.

A broad generalisation of Nica's $C^*$-algebras associated to quasi-lattice ordered groups has recently been introduced by Li \cite{Li2012}. He associates a number of $C^*$-algebras to discrete left cancellative semigroups. This generality is possible because of the importance of the right ideal structure of the semigroup; the full $C^*$-algebra $C^*(P)$ is generated by an isometric representation of $P$ and a family of projections associated to right ideals in $P$ satisfying a set of relations. As well as quasi-lattice ordered groups, Li's construction caters for the $ax+b$-semigroups over the rings of algebraic integers in number fields (see also \cite{cdl}). We will examine the $ax+b$-semigroup over $\Z$, which is the ring of algebraic integers in $\Q$.

A seemingly unrelated class of $C^*$-algebras has recently been discovered by Nekrashevych \cite{nek_jot,N2009}, namely those associated with self-similar group actions. 
The first example of a self-similar action was given by Grigorchuk \cite{Gri80}. As an infinite finitely-generated torsion group with intermediate growth, Grigorchuk's example solved a number of open problems, see \cite[p.14]{nek_book}. Since then a large number of interesting group actions have been shown to be self-similar and we refer the reader to Nekrashevych's book~\cite{nek_book} for further details.

A self-similar action $(G,X)$ consists of 
a group $G$ with a faithful action on the set $X^*$ of finite words on a finite set $X$; 
the action is self-similar in the sense that for each $g\in G$ and $x\in X$ 
there exists a unique $g|_x\in G$ such that $ g\cdot (xw) = (g\cdot x) (g|_x \cdot w)$ for all $w\in X^*$. 
Nekrashevych associated a Cuntz-Pimsner $C^*$-algebra to a self-similar action $(G,X)$ via generators and relations.
The algebra is generated by a unitary representation of $G$ and a collection of isometries associated to $X$, with  commutation relations modelled on the self-similarity relations. The Cuntz-Pimsner algebra contains copies of the full group $C^*$-algebra and the Cuntz algebra $\OO_{|X|}$. 
Since then, a universal Toeplitz-Cuntz-Pimsner algebra has been constructed that contains a generalised version of  Nekrashevych's algebras as a quotient~\cite{lrrw}. The self-similar commutation relations provide for an extremely simple generating set in both cases, and make the algebras particularly tractable. These commutation relations have been the inspiration for the results in this paper.

We identify a class of $C^*$-algebras that includes both those associated to quasi-lattice ordered groups and those associated to self-similar actions. We do this using a construction that was developed by G. Zappa in~\cite{Z} and J. Sz\'{e}p in~\cite{Szep1,Szep2,Szep3}. Given two groups, one can potentially impose a number of group-theoretic structures on their Cartesian product. In a direct product, both groups embed in the product as normal subgroups. In a semidirect product only one of the groups need be normal in the product. In a \ZS product of two groups neither group  need be normal in the product. \ZS products of semigroups were first described by Kunze in \cite{Kun}, and more recently Brin \cite{Brin} has described \ZS products in much broader generality. We examine a class of \ZS products of semigroups, and we associate two $C^*$-algebras to these semigroups.

We start with the \ZS product of two left-cancellative semigroups with identities. Following Li's construction from~\cite{Li2012}, we produce a full $C^*$-algebra. We give a new presentation of this full $C^*$-algebra via generators and relations. We then introduce a boundary quotient $C^*$-algebra, also with a presentation in terms of generators and relations. In some cases our results  reduce to known results (see Remark \ref{qlog_ex} for the quasi-lattice ordered group case). In other cases our results provide new, and more tractable, presentations of known algebras. Our construction also applies to new examples not covered by previous frameworks. As well as the $C^*$-algebras associated to quasi-lattice ordered groups and self-similar actions, we describe an example of a $C^*$-algebra associated to a self-similar action of a semigroup (see Sections \ref{subsec: the adding machine} and \ref{subsec: the adding machine C star C*}) and to products of self-similar actions.

The paper is organised as follows.
Section~\ref{sec: background} contains background material on the classes of semigroups we consider, and on Li's construction of the full $C^*$-algebra associated to discrete left-cancellative semigroups. In Section~\ref{sec: bowtie products} we recall the general \ZS product of semigroups, and we examine the examples of interest to us. In Section~\ref{sec: the generalised Toeplitz algebra} we give our alternative presentation of Li's full $C^*$-algebra via generators and relations. In Section~\ref{sec: the boundary quotient} we introduce the boundary quotient. We finish in Section~\ref{sec: examples} with an examination of the $C^*$-algebras associated to the examples of \ZS products introduced in Sections~\ref{subsec: baumslag-solitar groups}--\ref{subsec:products_SSAs}.

We thank Marcelo Laca for several interesting and helpful conversations about this work.

\section{Background}\label{sec: background}

In this section we present background material on semigroups and their $C^*$-algebras. 
Note that we consider only discrete left-cancellative semigroups satisfying an additional property on right common multiples as described in Definition~\ref{def: LCM}. We also outline Li's construction for associating a $C^*$-algebra to a left-cancellative semigroup from \cite{Li2012}.

\subsection{Semigroups}\label{subset: background - semigroups} 
All semigroups considered in this paper will have an identity, and hence are monoids. For a semigroup $P$, we write $P^*$ for the set of invertible elements.
Recall that $P$ is left cancellative if $pq=pr\Longrightarrow q=r$ for all $p,q,r\in P$.
We work with semigroups in which elements have right least common multiples in the following sense.

\begin{definition}\label{def: LCM}
Suppose $P$ is a discrete left-cancellative semigroup. We say $r\in P$ is a {\em right multiple} of $p\in P$ if there exists $q \in P$ such that $pq=r$. 
An element $r\in P$ is a {\em right least common multiple} (or {\em right LCM}) of $p$ and $q$ in $P$ if $r$ is a right common multiple of $p,q\in P$, and any other right common multiple of $p$ and $q$ is also a right multiple of $r$.
 We say $P$ is a {\em right LCM semigroup} if 
 any two elements with a right common multiple have a right least common multiple. 
\end{definition}

Even if they exist, least common multiples need not be unique. The following result will be obvious to experts but we couldn't find a reference, so we include it for completeness.

\begin{lemma}\label{lem:LCMs}
Suppose $P$ is a discrete left-cancellative semigroup and that $r\in P$ is a right LCM for $p,q\in P$. 
Then $s\in P$ is a right LCM for $p$ and $q$ if and only if $s=ru$ for some $u\in P^*$. 
\end{lemma}
\begin{proof}
Suppose $s \in P$ is a right LCM for both $p$ and $q$ in $P$. 
Since both $r,s$ are right least common multiples for $p,q\in P$, 
there exist $r', s'\in P$ such that $s=rr'$ and $r=ss'$. 
Suppose $e\in P$ is the identity in $P$. 
Then $se=s=rr'=(ss')r'=s(s'r')$. Since $P$ is left cancellative, we conclude $s'r'=e$. 
Similarly, $r's'=e$. Hence $r'\in P^*$  and $s=ru$ for some $u\in P^*$ as required.

Now suppose that $s=ru$ for some $u\in P^*$. Since $r$ is a right LCM for $p,q$, there exist $p',q'\in P$ such that $pp'=r=qq'$. Thus $pp'u =s= qq' u$, and hence $s$ is a right common multiple for $p,q$. To see that $s$ is a right LCM for $p,q$, suppose $t\in P$ is a right common multiple of $p,q$. So there exist $p'',q''\in P$ such that $pp''=t=qq''$. Since $r$ is a right LCM, there exist $r''$ such that $t=rr''$. Then $t=ruu^{-1}r''=st'$ for $t'=u^{-1}r''\in P$, and hence $t$ is a right multiple of $s$. Since $t$ was an arbitrary right common multiple of $p,q$, we conclude that $s$ is a right LCM for $p$ and $q$ in $P$.
\end{proof}


\begin{example}\label{QLO}
An example of a right LCM semigroup comes from the quasi-lattice ordered groups introduced by Nica in \cite{n}. Let $G$ be a discrete group and $P$ a subsemigroup of $G$ with $P^*=\{e\}$. Then $P$ induces a partial order on $G$ via $x\le y\Longleftrightarrow x^{-1}y\in P$. The pair $(G,P)$ is a {\em quasi-lattice ordered group} if any $x,y\in G$ which have a common upper bound in $P$ have a least upper bound $x\vee y\in P$. That is, if $(G,P)$ is quasi-lattice ordered, then $P$ is a right LCM semigroup with unique right LCMs.
\end{example}

\subsection{Semigroup $C^*$-algebras}\label{subset: semigroup C*-algebras} Let $P$ be a discrete left cancellative semigroup. We recall Li's construction of the $C^*$-algebra $C^*(P)$ from \cite{Li2012}. Given $X \subseteq P$ and $p \in P$ define
\[
 pX = \{px : x \in X\}\quad \text{and}\quad p\inv X = \{y\in P: py \in X\}.
\]
A set $X\subseteq P$ is called a right ideal if it is closed under right multiplication with any element of $P$. If $X$ is a right ideal, then so are $pX$ and $p\inv X$.

\begin{definition}[{\cite[p.4]{Li2012}}]
Let $\Jj(P)$ be the smallest family of right ideals of $P$ satisfying
\begin{enumerate}
 \item[(1)] $P, \varnothing \in \Jj(P)$;
 \item[(2)] $X \in \Jj(P)$ and $p \in P$ implies $pX$ and $p\inv X \in \Jj(P)$; and
 \item[(3)] $X, Y \in \Jj(P)$ implies $X \cap Y \in \Jj(P)$.
\end{enumerate}
The elements of $\Jj(P)$ are called constructible right ideals.
\end{definition}

\begin{remark}\label{rem: constructible right ideals}
The general form of a constructible right ideal is given in \cite[Equation~(5)]{Li2012}. However, the semigroups of interest to us are all right LCM semigroups, and in this case the constructible right ideals are precisely the principal right ideals; that is, $\JJ(P)=\{pP:p\in P\}$. Notice that if $P$ is an arbitrary right LCM semigroup, then principal right ideals associated to distinct $p,q\in P$ are not necessarily distinct, as is the case in the example discussed in Section~\ref{subsec: self-similar actions}. If $(G,P)$ is a quasi-lattice ordered group, then $p\not=q\Longrightarrow pP\not= qP$.  
\end{remark}

We can now give Li's definition of the full semigroup $C^*$-algebra for $P$.

\begin{definition}[{\cite[Definition 2.2]{Li2012}}]\label{def:semigroupalgebra}
Suppose $P$ is a discrete left-cancellative semigroup. Let $C^*(P)$ be the universal $C^*$-algebra generated by isometries $\{v_p : p\in P\}$ and projections $\{e_X : X \in \Jj(P)\}$ satisfying
\begin{enumerate}
 \item[(1)] $v_pv_q = v_{pq}$;
 \item[(2)] $v_p e_X v_p^* = e_{pX}$;
 \item[(3)] $e_P = 1$ and $e_\varnothing = 0$; and
 \item[(4)] $e_X e_Y = e_{X \cap Y}$,
\end{enumerate}
for all $p,q\in P$ and $X,Y\in\JJ(P)$.
\end{definition}

\begin{example}\label{QLOC*}
When $(G,P)$ is quasi-lattice ordered, Nica \cite{n} constructed a $C^*$-algebra $C^*(G,P)$ which is universal for isometric representations $V$ of $P$ satisfying
\begin{equation}\label{eq: nica cov}
V_p^*V_q=
\begin{cases}
V_{p^{-1}(p\vee q)}V_{q^{-1}(p\vee q)}^* & \text{if $p\vee q<\infty$}\\
0 & \text{if $p\vee q=\infty$.}
\end{cases}
\end{equation} 
Li showed in \cite[Section~2.4]{Li2012} that $C^*(P)\cong C^*(G,P)$.
\end{example}

\section{Zappa-Sz\'{e}p products}\label{sec: bowtie products}

The \ZS product of two groups was developed by G. Zappa in \cite{Z} and J. Sz\'{e}p in \cite{Szep1,Szep2,Szep3}. Brin \cite{Brin}  described \ZS products in a much broader generality, including \ZS products of semigroups. The following definition is given in \cite[Lemma 3.13(xv)]{Brin}.

\begin{definition}\label{def: the external bowtie}
Suppose $A$ and $U$ are semigroups with identities $e_A$ and $e_U$, respectively. Assume the existence of maps $A \times U \to U$ given by $(a,u) \mapsto a \cdot u$, and $A \times U \to A$ given by $(a,u) \mapsto a|_u$, satisfying
 \begin{tabbing}
 \,\,\,(B1) $e_A \cdot u = u$;\hspace{4cm} \=  (B5) $a \cdot (uv) = (a \cdot u)(a|_u \cdot v)$; \\
 \,\,\,(B2) $(a b) \cdot u = a \cdot (b \cdot u)$; \>(B6) $a|_{uv} = (a|_u)|_v$;\\
 \,\,\,(B3) $a \cdot e_U = e_U$; \> (B7) $e_A|_{u}=e_A$; and\\
 \,\,\,(B4) $a|_{e_U}=a$;  \> (B8) $(a b)|_u = a|_{b \cdot u} b|_u$.
\end{tabbing}
The external \ZS product $U \bowtie A$ is the cartesian product $U \times A$ with multiplication given by
\begin{equation}\label{bowtie_action}
(u,a)(v,b) = (u(a \cdot v), (a|_v)b)
\end{equation}
For each $a\in A$ and $u\in U$ we call $a|_u$ the {\em restriction} of $a$ to $u$, and $a\cdot u$ the {\em action} of $a$ on $u$. 
\end{definition}

The following result \cite[Lemma 3.9]{Brin} describes the internal \ZS product.

\begin{prop}\label{prop: internal bowtie}
Suppose $P$ is a semigroup with identity. Suppose that  $U,A \subseteq P$ are subsemigroups  of $P$ with $U \cap A =\{e\}$ and such that for all $p \in P$ there exists unique $(u,a) \in U \times A$ such that $p=u a$. For  $a\in A$ and $u\in U$ define $a \cdot u \in U$ and $a|_u \in A$ by $a u = (a \cdot u) a|_u$. The action and restriction maps so defined satisfy conditions (B1)--(B8)  and $P \cong U \bowtie A$.
\end{prop}

The following result gives sufficient conditions for $U\bowtie A$ to be a right LCM semigroup.

\begin{lemma}\label{Davs delight}
Suppose $U$ and $A$ are left cancellative semigroups with maps $(a,u) \mapsto a \cdot u$ and $(a,u) \mapsto a|_u$ satisfying (B1)--(B8) of Definition~\ref{def: the external bowtie}. Moreover, suppose $U$ is a right LCM semigroup, $\Jj(A)$ is totally ordered by inclusion, and $u \mapsto a\cdot u$ is a bijective map from $U$ to $U$ for each $a\in A$. Then $U \bowtie A$ is a right LCM semigroup.
\end{lemma}

\begin{proof}
We first show that $U\bowtie A$ is left cancellative. Suppose $(u,a)(v,b)=(u,a)(w,c)$. Then $u(a\cdot v)=u(a\cdot w)$ and $a|_vb=a|_wc$. Since $U$ is left cancellative, we have $a\cdot v=a\cdot w$. Since the action of $a$ is injective, we have $v=w$. Then $a|_v=a|_w$, and because $A$ is left cancellative we have  $b=c$. So $(v,b)=(w,c)$.

Now suppose $(u,a), (v,b) \in U \bowtie A$ have a right common multiple; so there exists elements $(u',a'), (v',b') \in U \bowtie A$ such that $(u,a)(u',a') = (v,b)(v',b')$. In particular, $u(a \cdot u')=v(b \cdot v') \in U$. Since $U$ is a right LCM semigroup, $u$ and $v$ have a right LCM $w\in U$. Fix $u''$ and $v''$ such that $u u''=w=vv''$. Since we have assumed $(u,a) \mapsto a \cdot u$ is surjective for fixed $a$, there exists $x,y \in U$ such that $a \cdot x = u''$ and $b \cdot y = v''$. Since $\Jj(A)$ is totally ordered, we can assume without loss of generality that $a|_x A \cap b|_y A = b|_y A$. Fix $a''$ such that $a|_x a'' = b|_y$. Then
\[
(u,a)(x,a'')= (u(a \cdot x), a|_x a'') = (uu'',b|_y)= (w,b|_y),
\]
and
\[
(v,b)(y,e)=(v(b \cdot y),b|_y )=(vv'',b|_y)=(w,b|_y).
\]
So $(w,b|_y)$ is a right common multiple of $(u,a)$ and $(v,b)$.

To see that $(w,b|_y)$ is a right LCM, suppose $(u,a)(s,c)=(v,b)(t,d)$. Then $u(a \cdot s) = v(b \cdot t) = ww'$ for some $w' \in U$. Since $(u,a) \mapsto u \cdot a$ is surjective, there exists $t' \in U$ such that $w' = b|_{y} \cdot t'$. Then
\begin{align*}
v(b \cdot t)&=ww'=vv''w'=v(b\cdot y)(b|_y \cdot t')=v(b\cdot(yt'))
\end{align*}
so $b \cdot t = b \cdot (yt')$. Since $(a,u) \mapsto a \cdot u$ is injective for fixed $a$, this implies that $t=yt'$.
Hence $(v,b)(t,d) = (w,b|_y)(t',d) = (u,a)(s,c)$, and so $(w,b|_y)$ is a right LCM.
\end{proof}

\begin{remark}\label{rem: the lcm formula}
Calculations in the above proof produce some useful observations about  semigroups~$U$ and~$A$ satisfying the hypothesis of Lemma~\ref{Davs delight}.
Firstly,
\[
(u,a)U\bowtie A\cap (v,b)U\bowtie A=\varnothing \Longleftrightarrow uU\cap vU=\varnothing.
\]
Secondly, a right LCM can be rapidly identified in the following cases.
\begin{enumerate}
\item[(a)] If $u,v\in U$ have right LCM $z\in U$, then $(z,e_A)$ is a right LCM of $(u,e_A)$ and $(v,e_A)$.
\item[(b)] For $a\in A$ and $u\in U$ a right LCM of $(e_U,a)$ and $(u,e_A)$ is 
\begin{equation*}\label{eq: LCM of special elts}
(u,a|_z)=(e_U,a)(z,e_A)=(u,e_A)(e_U,a|_z),
\end{equation*}
where $z$ is the unique element in $U$ such that $a\cdot z=u$.
\end{enumerate}
\end{remark}

Perhaps surprisingly, a number of interesting examples are \ZS products of the form $U\bowtie A$ where $U$ and $A$ satisfy the hypotheses of Lemma~\ref{Davs delight}.
We now examine some of them.

\subsection{Baumslag-Solitar groups}\label{subsec: baumslag-solitar groups} Let $c$ and $d$ be positive integers. The Baumslag-Solitar group $BS(c,d)$ is the group with presentation $\langle a,b \mid a b^c = b^d a \rangle$. We denote by $BS(c,d)^+$ the subsemigroup of $BS(c,d)$ generated by $a$ and $b$.

By \cite[Chapter IV, Theorem 2.1]{LS}, every element $p \in BS(c,d)^+$ admits a unique normal form
\[
 p = b^{\alpha_1}ab^{\alpha_2}a\dots b^{\alpha_n}ab^{\beta},
\]
where each $\alpha_i \in \{0,\dots,d-1\}$ and $\beta\in\NN$. Consider the following subsemigroups of $BS(c,d)^+$:
\[
 U := \langle e, a, ba, \dots, b^{d-1}a \rangle\quad\text{and}\quad A := \langle e, b \rangle.
\]
We have $U \cap A = \{e\}$. We can also see from the normal form that each $p=b^{\alpha_1}ab^{\alpha_2}a\dotsb^{\alpha_n}ab^{\beta} \in BS(c,d)^+$ can be written uniquely in $UA$ as the product of $b^{\alpha_1}ab^{\alpha_2}a\dotsb^{\alpha_n}a\in U$ and $b^\beta\in A$. So Proposition \ref{prop: internal bowtie} implies that $BS(c,d)^+ \cong U\bowtie A$. On generators, the action and restriction maps satisfy
\begin{equation}\label{eq: BS action}
 b \cdot b^k a = 
 \begin{cases}
 b^{k+1} a & \text{ if $k < d-1$}\\
  a & \text{ if $k = d-1$}
  \end{cases}
  \end{equation}
and
\begin{equation}\label{eq: BS restriction}
 b|_{b^k a} = 
  \begin{cases}
e & \text{ if $k < d-1$}\\
  b^c & \text{ if $k = d-1$.}
  \end{cases}
\end{equation}
The subsemigroup $U$ is the free semigroup on $d$ generators, and hence is right LCM. The subsemigroup $A$ is left cancellative, and for each $\alpha,\beta\in\N$ we have $b^{\max\{\alpha,\beta\}}A\subseteq b^{\min\{\alpha,\beta\}}A$, and so $\JJ(A)$ is totally ordered by inclusion. It follows from \eqref{eq: BS action} that the action of each $b^\beta\in A$ on $U$ is bijective. So the hypotheses of Lemma~\ref{Davs delight} are satisfied.

\subsection{The semigroup $\nxn$}\label{subsec: n by n times}

Consider the semigroups $\N = \{n \in \Z \mid n \geq 0\}$ under addition, $\N^\times = \{n\in \Z \mid n\ge 1\}$ under multiplication, and $\Q_+^* = \{q\in\Q \mid q>0\}$ under multiplication. Consider the semidirect product $\Q\rtimes\Q_+^*$, where
\[
(r,a)(q,b)= (r+aq, ab) \quad\text{ for } r,q \in \Q \text{ and } a,b \in \Q_+^*.
\]
The semidirect product $\nxn$ is a subsemigroup of $\qxq$, and the pair $(\qxq,\nxn)$ is quasi-lattice ordered \cite[Proposition~2.1]{lr}. We will now describe $\nxn$ as a \ZS product.

Consider the following subsemigroups of $\N\rtimes\N^\times$:
\[
U := \{(r,x) : x \in \N^\times, 0\le r \le x-1\}\quad\text{and}\quad A := \{(m,1) : m \in \N\}.
\]
We have $U\cap A=\{(0,1)\}$, which is the identity of $\nxn$. We can write each $(m,a)\in\nxn$ uniquely as a product in $UA$ via
\[
(m,a) = \Big( m\,(\Mod\, a) ,a \Big) \Big(\frac{m-(m\,(\Mod\, a))}{a},1 \Big).
\]
So the hypotheses of Proposition~\ref{prop: internal bowtie} are satisfied, and hence $\nxn\cong U\bowtie A$. The action and restriction maps are given by
\begin{equation}\label{eq: n by n times a and r}
(m,1) \cdot (r,x) = ((m+r)\,(\Mod\, x),x)\quad\text{and}\quad (m,1)|_{(r,x)} = \Big( \frac{m+r - ((m+r)\,(\Mod\, x))}{x},1 \Big).
\end{equation}
Both $U$ and $A$ are subsemigroups of a left cancellative semigroup $\nxn$, and hence are both left cancellative. For each $m,n\in\N$ we have $(\max\{m,n\},1)A\subseteq (\min\{m,n\},1)A$, and so $\JJ(A)$ is totally ordered by inclusion. The next result shows that $U$ is right LCM.

\begin{lemma}\label{lem: U for n by n times is right lcm}
Consider the subsemigroup $U$ of $\nxn$ described above. Let $(r,x),(s,y)\in U$. If $(r+x\N)\cap (s+y\N)\not=\varnothing$, then the right LCM of $(r,x)$ and $(s,y)$ is $(l,\lcm(x,y))$, where $l$ is the least element of $(r+x\N)\cap (s+y\N)$. If $(r+x\N)\cap (s+y\N)=\varnothing$, then $(r,x)$ and $(s,y)$ have no common multiple.
\end{lemma}

\begin{proof}
Let $(r,x),(s,y)\in U$. We know from \cite[Remark~2.3]{lr} that 
\[
(r,x)\vee (s,y)=
\begin{cases}
(l,\lcm(x,y)) & \text{if $(r+x\N)\cap (s+y\N)\not=\varnothing$}\\
\infty & \text{if $(r+x\N)\cap (s+y\N)=\varnothing$,}
\end{cases}
\]
where $l$ is the least element of $(r+x\N)\cap (s+y\N)$. Suppose $(r+x\N)\cap (s+y\N)\not=\varnothing$, and let $j,k\in\N$ with $l=r+xj=s+yk$. Also let $x',y'\in\N^\times$ with $\lcm(x,y)=xx'=yy'$. Then $(l,\lcm(x,y))=(r,x)(j,x')=(s,y)(k,y')$. For $(l,\lcm(x,y))$ to be an element of $U$, it suffices to show that $j<x'$, which would also imply $k<y'$. Suppose for contradiction that $j\ge x'$. Then we must also have $k\ge y'$. Write $j=x'+j'$ and $k=y'+k'$ for some $j',k'\in\N$. Then
\[
r+xj=s+yk\Longleftrightarrow r+ xx'+xj'=s+yy'+yk'\Longleftrightarrow r+xj'=s+yk',
\]
which contradicts that $l$ is the least element of $(r+x\N)\cap (s+y\N)$. Hence we must have $j<x'$, and the result follows.
\end{proof}

To check that $U$ and $A$ satisfy all the hypotheses of Lemma~\ref{Davs delight}, it remains to check that for each $(m,1)\in A$, the map $u \mapsto (m,1) \cdot u$ is a bijection on $U$. Fix $(m,1)\in A$. We have
\begin{align*}
(m,1) \cdot (r,x) = (m,1)\cdot (s,y) &\Longleftrightarrow ((m+r)\,(\Mod\, x),x)=((m+s)\,(\Mod\, y),y)\\
&\Longleftrightarrow x=y\text{ and } r=s\\
&\Longleftrightarrow (r,x)=(s,y).
\end{align*}
So the action of $(m,1)$ is injective. To see that the action of $(m,1)$ is surjective, fix $(r,x)\in U$. Let
\[
a:=
\begin{cases}
r-\big(m\,(\Mod\, x)\big) & \text{if $r\ge m\,(\Mod\, x)$}\\
x-(m\,(\Mod\, x)-r) & \text{if $r< m\,(\Mod\, x)$.}
\end{cases}
\]
Then $(a,x)\in U$ and $(m,1)\cdot (a,x)=(r,x)$. So the action of $(m,1)$ is surjective. We can now apply Lemma~\ref{Davs delight} to see that $\nxn\cong U\bowtie A$ is a right LCM semigroup. Of course, we know from \cite{lr} that $\nxn$ is quasi-lattice ordered, which is stronger than right LCM. But we need to know each pair $(U, A)$ in our examples satisfy the hypotheses of Lemma~\ref{Davs delight} so we can apply our later results, as noted in the following remark. 

\begin{remark}\label{rem: }
In general there is no unique way of decomposing a semigroup into a \ZS product, as can be illustrated with the semigroup $\nxn$. In addition to the decomposition $\nxn\cong U\bowtie A$ described above, we have $\nxn\cong\N\bowtie\N^\times$, where $a\cdot m=am$ and $a|_m=a$. Also note that even though $U\bowtie A$ satisfies the hypotheses of Lemma~\ref{Davs delight}, the \ZS product $\N\bowtie\N^\times$ does not. So while the $C^*$-algebras $C^*(U\bowtie A)$ and $C^*(\N\bowtie\N^\times)$ as described in Section~\ref{sec: the generalised Toeplitz algebra} are isomorphic, the presentation given in Theorem~\ref{thm: main c star theorem} only applies to $C^*(U\bowtie A)$.
\end{remark}

\subsection{The semigroup $\zxz$}\label{subsec: the semigroup zxz}

Denote $\Z^\times:=\Z\setminus\{0\}$. The $ax+b$-semigroup over $\Z$ is the semidirect product $\zxz$, where $(m,a)(n,b)=(m+an,ab)$. Define subsemigroups
\[
U=\{(r,x):x\ge 1,\,0\le r<x\}\quad\text{and}\quad A=\Z\times\{1,-1\}.
\]
So $U$ is the same as the semigroup appearing in Section~\ref{subsec: n by n times}, and $A$ is a group. We have $U\cap A=\{(0,1)\}$, which is the identity of $\zxz$. For each $(m,a)\in\zxz$ we can uniquely write
\[
(m,a)=
\Big( m\,(\Mod\, |a|) ,|a| \Big) \Big(\frac{m-(m\,(\Mod\, |a|))}{|a|},\frac{a}{|a|} \Big)\in UA.
\]
So we can apply Proposition~\ref{prop: internal bowtie} to see that $\zxz\cong U\bowtie A$. The action and restriction maps are given by
\[
(m,j) \cdot (r,x) = ((m+jr)\,(\Mod\, x),x)\quad\text{and}\quad (m,j)|_{(r,x)} = \Big( \frac{m+jr - ((m+jr)\,(\Mod\, x))}{x},j \Big),
\]
for $j\in\{1,-1\}$.

Both $U$ and $A$ are left cancellative. Since $A$ is a group, $\JJ(A)=\{A\}$ is trivially totally ordered. In Lemma~\ref{lem: U for n by n times is right lcm} we proved that $U$ is right LCM. To show that $U$ and $A$ satisfy the hypotheses of Lemma~\ref{Davs delight}, we just need to check that the action of each fixed $(m,j)$ is bijective. 

Fix $(m,j)\in A$. We have
\begin{align*}
(m,j) \cdot (r,x) = (m,j)\cdot (s,y) &\Longleftrightarrow ((m+jr)\,(\Mod\, x),x)=((m+js)\,(\Mod\, y),y)\\
&\Longleftrightarrow x=y\text{ and } (m+jr)-(m+js)\in x\Z\\
&\Longleftrightarrow x=y\text{ and } j(r-s)\in x\Z\\
&\Longleftrightarrow x=y\text{ and } r=s\\
&\Longleftrightarrow (r,x)=(s,y)\\
\end{align*}
So the action of $(m,j)$ is injective. To see that the action of $(m,j)$ is surjective, fix $(r,x)\in U$. Let
\[
s:=
\begin{cases}
\Big(j\big(r-\big(m\,(\Mod\, x)\big)\big)\Big)(\Mod\, x) & \text{if $r\ge m\,(\Mod\, x)$}\\
\Big(j\big(x-(m\,(\Mod\, x)-r)\big)\Big)(\Mod\, x) & \text{if $r< m\,(\Mod\, x)$.}
\end{cases}
\]
Then $(s,x)\in U$ and $(m,j)\cdot (s,x)=(r,x)$. So the action of $(m,j)$ is surjective.

\subsection{Self-similar actions}\label{subsec: self-similar actions}

Let $X$ be a finite alphabet. We write $X^n$ for the set of words of length $n$ in~$X$ and $X^*:=\bigcup_{n=0}^\infty X^n$. The set $X^*$ has a geometric realisation as a homogenous rooted tree with root $\varnothing$; that is, vertices in the tree are associated with words in $X^*$ and for each $w \in X^*$ there is an edge from $w$ to $wx$ for all $x \in X$. We will be considering subgroups of the automorphism group on the rooted tree $X^*$.

A faithful action of a group $G$ on $X^*$ is \emph{self-similar} if for every $g \in G$ and $x\in X$, there exist unique $g|_x \in G$ such that
\begin{equation}\label{SSA_cond}
 g\cdot (xw) = (g\cdot x) (g|_x \cdot w).
\end{equation}
The group element $g|_x$ is called the \emph{restriction of $g$ to $x$}. Notice that restriction extends to words of finite length by iteration. When there is a self-similar action of $G$ on $X^*$, the pair $(G,X)$ is called a {\em self-similar action}.

\begin{lemma}[{\cite[\S1.3]{nek_book}}] \label{lem:props}
Suppose $(G,X)$ is a self-similar action.
\begin{enumerate}
\item For  $g, h\in G$ and $v,w\in X^*$, we have
\[
g|_{vw}=(g|_{v})|_{w}, \quad
gh|_{v} = g|_{h\cdot v} h|_{v}, \quad\text{and}\quad
g|_{v}^{-1} = g^{-1}|_{g\cdot v}.
\]
\item For every $g \in G$, the map $g:X^n \to X^n$ given by $w \mapsto g \cdot w$ is bijective.
\end{enumerate}
\end{lemma}

We now recall Lawson's \cite{Law} description of self-similar actions as \ZS product semigroups. The tree $X^*$ is naturally a semigroup with concatenation of words as the operation and identity $\varnothing$. Restriction gives a map $G \times X^* \rightarrow G$ by $(g,w)=g|_w$ and the group action defines a map $G \times X^* \rightarrow X^*$ by $(g,w) = g \cdot w$. We need to show that these maps satisfy (B1)--(B8) in Definition \ref{def: the external bowtie}. Since the map $(g,w) = g \cdot w$ is a group action (B1)--(B3) are automatic. The relation (B4) follows from the fact that $g \cdot (\varnothing u)=g \cdot u$. Equation (B5) is the self-similar action condition \eqref{SSA_cond}. Lemma \ref{lem:props} (1) gives (B6) and (B8). That the group action is faithful implies (B7). Therefore, $X^* \bowtie G$ is an external Zappa-Sz\'{e}p product with multiplication given by
\begin{equation}\label{SSA_bowtie}
 (x,g)(y,h) = (x(g\cdot y),g|_y h) \in X^* \times G.
\end{equation}
We note that Lemma \ref{lem:props} (2) implies that for every $g \in G$, the map $w \mapsto g \cdot w$ is bijective. Lawson \cite{Law} goes on to prove the following.

\begin{thm}[{\cite[Propositions 3.5 and 3.6]{Law}}]
Let $(G,X)$ be a self-similar action. With the above product, $X^* \bowtie G$ is a right LCM semigroup with identity $(\varnothing,e)$ and the pair $X^*$ and $G$ satisfy the hypotheses of Lemma \ref{Davs delight}. Moreover, the poset of principal right ideals of $X^*\bowtie G$ is order isomorphic to the poset of principal right ideals of $X^*$.
\end{thm}

\begin{remark}
A semigroup $S$ is called a {\em left Rees monoid} if $S$ is left cancellative, each principal right ideal $sS$ is properly contained in only a finite number of principal right ideals, and $sS \cap tS \neq \varnothing$ implies either $sS \subseteq tS$ or $tS \subseteq sS$. In \cite{Law}, Lawson showed that a semigroup is a left Rees semigroup if and only if it is a \ZS product of a free semigroup by a group acting self-similarly.
\end{remark}

\subsection{The adding machine}\label{subsec: the adding machine} Consider the alphabet $X = \{0,1,\dots,n-1\}$ for some $n\in \N$. There is a self-similar action of $\Z = \langle e,\gamma\rangle$ on the tree $X^*$, where the action and restriction of $\gamma$ on a letter $k\in X$ is given by
\[
 \gamma \cdot k = (k+1)\,(\Mod n)
\]
and
\[
 \gamma|_{k} =
 \begin{cases}
 e & \text{ if $k<n-1$}\\ 
 \gamma & \text{ if $k = n-1$.}
 \end{cases}
\]
The self-similar action $(\Z,X)$ is commonly known as the \emph{adding machine}, or \emph{odometer}. Since the subsemigroup $\N \subset \Z$ is invariant under the restriction map, we may form the \ZS product $X^* \bowtie \N$. Looking at the action and restriction described above, and the action \eqref{eq: BS action} and restriction \eqref{eq: BS restriction} for $\BS(c,d)^+$ with $c=1$ and $d=n$, we see that $X^*\bowtie \N$ is isomorphic to $\BS(1,n)^+$. 

We can also describe $\BS(1,n)^+$ as a subsemigroup of the \ZS product $U\bowtie A$ isomorphic to $\nxn$ from Section~\ref{subsec: n by n times}. Consider the free subsemigroup of $\N\rtimes\N^\times$
\[
 U_n = \langle(0,1),(0,n),(1,n),\dots,(n-1,n)\rangle
\]
and $A=\{(m,1):m\in\N\}$. We see from the action formula given in \eqref{eq: n by n times a and r} that $U_n$ is invariant under the action of $A$, and the \ZS product $U_n \bowtie A$ is isomorphic to $\BS(1,n)^+$.

\subsection{Products of self-similar actions.}\label{subsec:products_SSAs} Suppose $X$ and $Y$ are finite alphabets, and $G$ is a group which acts self-similarly on both $X$ and $Y$. Assume the existence of a bijective map $\theta : Y\times X\to X\times Y$. For each $(y,x)\in Y\times X$ we denote by $\theta_X(y,x)\in X$ and $\theta_Y(x,y)\in Y$ the unique elements satisfying $\theta(y,x) = (\theta_X(y,x),\theta_Y(y,x))$. Let $\F^+_\theta$ denote the semigroup generated by $X\cup Y \cup \{e\}$ with relations $yx = \theta_X(y,x)\theta_Y(y,x)$ for all $x\in X$ and $y\in Y$. Note that these semigroups are the $2$-graphs with a single vertex studied in \cite{dpy, DY2009}.

Repeated applications of the bijection $\theta$ implies that every element $z\in\F^+_\theta$ admits a normal form $z = vw$ where $v \in X^*$ and $w \in Y^*$. The self-similar actions of $G$ on $X$ and $Y$ induce maps $G\times\F_\theta^+\to \F_\theta^+$ and $G\times\F_\theta^+\to G$ given by
\begin{equation}\label{eq: maps for product of seas}
(g,z)\mapsto g \cdot z: = (g\cdot v)(g|_v \cdot w)\quad\text{and}\quad(g,z)\mapsto g|_{z} := (g|_v)|_w,
\end{equation}
respectively. The following result gives necessary and sufficient conditions for the maps in \eqref{eq: maps for product of seas} to give a \zsp $\F_\theta^+\bowtie G$.

\begin{prop} \label{prop:productssa}
The maps given in \eqref{eq: maps for product of seas} induce a Zappa-Sz\'ep product semigroup $\F^+_\theta \bowtie G$ if and only if for all $g\in G$, $x\in X$ and $y\in Y$ we have
\[
 \theta_X(y,x) = g\inv \cdot \theta_X(g\cdot y,g|_y\cdot x)\quad\text{and}\quad \theta_Y(y,x) = g|_{\theta_X(y,x)}\inv \cdot \theta_Y(g\cdot y, g|_y \cdot x).
\]
\end{prop}

\begin{proof} For the forward direction, suppose $\F^+_\theta\bowtie G$ is a Zappa-Sz\'ep product, and fix $g\in G$, $x\in X$ and $y\in Y$. Then
\begin{align}
 \theta_X(y,x)\theta_Y(y,x) &= g\inv\cdot (g\cdot(\theta_X(y,x)\theta_Y(y,x))) \notag \\
 &= g\inv \cdot (g\cdot (yx)) \notag \\
 &= g\inv \cdot ((g\cdot y)(g|_y\cdot x)) \notag \\
 &= g\inv \cdot (\theta_X(g\cdot y,g|_y\cdot x)\theta_Y(g\cdot y,g|_y\cdot x)) \notag \\
 &= (g\inv \cdot \theta_X(g\cdot y,g|_y\cdot x)) (g\inv|_{\theta_X(g\cdot y,g|_y\cdot x)} \cdot \theta_Y(g\cdot y,g|_y\cdot x)). \label{g_inv_formula}
\end{align}
In particular, we see that
\[
 \theta_X(y,x) = g\inv \cdot \theta_X(g\cdot y,g|_y\cdot x),
\]
and $g\cdot \theta_X(y,x) = \theta_X(g\cdot y,g|_y\cdot x)$ so that $g|_{\theta_X(y,x)}\inv = g\inv|_{\theta_X(g\cdot y,g|_y\cdot x)}$ by the third identity in Lemma \ref{lem:props} (1). From \eqref{g_inv_formula}, we also see that
\[
 \theta_Y(y,x) = g|_{\theta_X(y,x)}\inv \cdot \theta_Y(g\cdot y,g|_y\cdot x)
\]
as required.

Conversely, suppose $\theta$ satisfies the above relations. We must show that conditions (B1) -- (B8) of Definition \ref{def: the external bowtie} are satisfied. We check (B5) and leave the remaining computations to the reader. It is enough to verify (B5) on elements $yx\in\F^+_\theta$ where $y \in Y$ and $x\in X$. We compute
\begin{align*}
 g\cdot (yx) &= g\cdot(\theta_X(y,x)\theta_Y(y,x)) \\
 &= (g\cdot\theta_X(y,x)) (g|_{\theta_X(y,x)}\cdot\theta_Y(y,x)) \\
 &= \theta_X(g\cdot y,g|_y\cdot x)\theta_Y(g\cdot y,g|_y\cdot x) \\
 &= (g\cdot y) (g|_y\cdot x),
\end{align*}
as required.
\end{proof}

\begin{remark}\label{rem: product of ssa satisfies conditions}
The semigroup $\F_\theta^+$ is left cancellative by the unique 
factorisation 
property of $k$-graphs, but it is not right LCM in general.
However, there are interesting examples, 
such as Example \ref{exm:productaddingmachines} below,
for which $\F_\theta^+$ is right LCM.
 Since $G$ is a group, the other hypotheses of Lemma~\ref{Davs delight} are automatically satisfied. So if $\F_\theta^+$ is right LCM, then $\F_\theta^+\bowtie G$ is a right LCM semigroup.
\end{remark} 

\begin{example} \label{exm:productaddingmachines}
In this example we use adding machine actions on two alphabets to
induce a self-similar action of~$\Z$ on a~2-graph with one vertex. 
Fix $m,n \geq 2$ and let $X:=\{x_0,x_1,\cdots,x_{m-1}\}$ and $Y:=\{y_0,y_1,\cdots y_{n-1}\}$. We can write the set $\{0,1,\dots,mn-1\}$ as
\[
 \{i+jm : 0\leq i \leq m-1, 0 \leq j \leq n-1\}\quad\text{and}\quad \{j+in : 0\leq i \leq m-1, 0 \leq j \leq n-1\}.
\]
It follows that there is a bijection $\{0,1,\dots,n-1\}\times\{0,1,\dots,m-1\}\to \{0,1,\dots,m-1\}\times\{0,1,\dots,n-1\}$ sending $(j,i)$ to the pair $(i',j')$ satisfying $j +in=i' + j'm$. This bijection induces a bijection $\theta : Y\times X \to X\times Y$ given by $\theta(y_j,x_i) = (x_{i'},y_{j'})$. The group of integers $\Z = \langle e,\gamma \rangle$ acts self-similarly on both $X^*$ and $Y^*$ by the adding machine action. Recall from Section~\ref{subsec: the adding machine} that the action is given by
\[
\gamma\cdot x_i =
\begin{cases}
x_{i+1} & \text{if $i<m-1$}\\
x_0 & \text{if $i=m-1$}
\end{cases}
\quad\text{and}\quad
\gamma\cdot y_j =
\begin{cases}
y_{j+1} & \text{if $j<n-1$}\\
y_0 & \text{if $j=n-1$,}
\end{cases}
\]
and the restriction is given by
\[
\gamma|_{x_i} =
\begin{cases}
e & \text{if $i<m-1$}\\
\gamma & \text{if $i=m-1$}
\end{cases}
\quad\text{and}\quad
\gamma|_{y_j} =
\begin{cases}
e & \text{if $j<n-1$}\\
\gamma & \text{if $j=n-1$.}
\end{cases}
\]
We leave it to the reader to show that the identities in \eqref{eq: 
maps for 
product of seas} hold in this example, and so we can apply 
Proposition 
\ref{prop:productssa} to get an integer action on a $2$-graph.

We claim that the semigroup $\F_\theta^+$ is right LCM if and only if 
$m$ and 
$n$ are 
coprime. For the reverse implication note that $\F_\theta^+$ is 
isomorphic to 
the subsemigroup of $\nxn$ generated by 
$\{(0,m),\dots,(m-1,m),(0,n),\dots,(n-1,n)\}$, and the arguments in 
the proof 
of Lemma~\ref{lem: U for n by n times is right lcm} show that this 
subsemigroup is 
right LCM. For the forward implication, suppose $m=pa$ and $n=pb$ for 
$p>1$. 
Then
\[
p+(0\times pb)=p+(0\times pa)\quad\text{and}\quad p+(a\times pb)=p+(b\times pa),
\]
and hence the elements $y_px_0=x_py_0$ and $y_px_a=x_py_b$ are 
incomparable 
right common multiples which cannot be larger than any other common 
multiple.

\end{example}

\section{The $C^*$-algebra $C^*(U\bowtie A)$}\label{sec: the generalised Toeplitz algebra}

In this section we will assume that $U$ and $A$ are semigroups satisfying the hypotheses of Lemma~\ref{Davs delight}; so $U\bowtie A$ 
is a right LCM semigroup. Consider the $C^*$-algebra $C^*(U\bowtie A)$ obtained by applying Li's construction as described in Section~\ref{subset: semigroup C*-algebras} to $U\bowtie A$. In this case $C^*(U\bowtie A)$ is the universal $C^*$-algebra generated by isometries $\{v_{(u,a)} \mid (u,a)\in U\bowtie A\}$ and projections $\{e_{(u,a)}:=e_{(u,a)U\bowtie A} \mid (u,a)\in U\bowtie A\}\cup\{e_\varnothing\}$ satisfying
\begin{enumerate}
 \item[(L1)] $v_{(u,a)}v_{(w,b)} = v_{(u,a)(w,b)}$;
 \item[(L2)] $v_{(u,a)} e_{(w,b)} v_{(u,a)}^* = e_{(u,a)(w,b)}$;
 \item[(L3)] $e_{(e_U,e_A)} = 1$ and $e_\varnothing = 0$; and
 \item[(L4)] $$e_{(u,a)} e_{(w,b)} = 
 \begin{cases}
 e_{(z,c)} & \text{if $(u,a)U\bowtie A\cap (w,b)U\bowtie A=(z,c)U\bowtie A$}\\
 0 & \text{if $(u,a)U\bowtie A\cap (w,b)U\bowtie A=\varnothing$.}
 \end{cases}$$
\end{enumerate}

Notice that (L2) and (L3) imply that 
\begin{equation}\label{eq: Li range projections}
v_{(u,a)} v_{(u,a)}^* = e_{(u,a)}\quad\text{for all $(u,a)\in U\bowtie A$}.
\end{equation}

The main result of this section is to give an alternative presentation of $C^*(U\bowtie A)$ in terms of isometric representations of the individual semigroups $U$ and $A$. First we need the notion of a covariant representation of a right LCM semigroup.

\begin{definition}\label{def: covariance}
Let $P$ be a right LCM semigroup. A {\em covariant representation} of $P$ in a $C^*$-algebra $B$ is an isometric representation $t$ satisfying
\[
t_p^*t_q=
\begin{cases}
t_{p'}t_{q'}^* & \text{if $pP\cap qP=rP$ and $pp'=qq'=r$}\\
0 & \text{if $pP\cap qP=\varnothing$.}
\end{cases}
\]
Since the right LCM of two elements is not in general unique, we need to check that the expression on the right-hand side is well defined. 
\end{definition}

\begin{lemma}\label{lem: well defined covariance}
Let $P$ be a right LCM semigroup, $p,q\in P$ with $pP\cap qP\not=\varnothing$, and $t$ an isometric representation of $P$. Suppose $p', q',r\in P$ with $pP\cap qP=rP$ and $pp'=qq'=r$, and $p'',q'',s\in P$ with $pP\cap qP=sP$ and $pp''=qq''=s$. Then $t_{p'}t_{q'}^*=t_{p''}t_{q''}^*$.
\end{lemma}

\begin{proof}
If $r=s$, then left cancellativity gives $p'=p''$ and $q'=q''$, and the result follows. So suppose $r\not= s$. Since $rP=sP$, there must be $u\in P^*$ with $u\not= e$ and $r=su$. Since $t_{u^{-1}}=t_u^*t_ut_{u^{-1}}=t_u^*$ (so $t_u$ is a unitary), we have
\[
t_{p'}t_{q'}^*=t_p^*t_p t_{p'}t_{q'}^*t_q^*t_q=t_p^*t_{pp'}t_{qq'}^*t_q=t_p^*t_{r}t_{r}^*t_q=t_p^*t_{su}t_{su}^*t_q=t_p^*t_{s}t_{s}^*t_q=t_p^*t_{pp''}t_{qq''}^*t_q=t_{p''}t_{q''}^*.\qedhere
\]
\end{proof}

We now state the main result.

\begin{thm}\label{thm: main c star theorem}
Suppose $U$ and $A$ are semigroups with maps $(a,u) \mapsto a \cdot u$ and $(a,u) \mapsto a|_u$ satisfying (B1)--(B8) of Definition~\ref{def: the external bowtie}. Moreover, suppose $U$ is a right LCM semigroup, $A$ is left cancellative with $\Jj(A)$ totally ordered by inclusion, and for each $a\in A$, the map $u \mapsto a\cdot u$ is bijective. Let $\AA$ be the universal $C^*$-algebra generated by an isometric representation $s$ of $A$ and a covariant representation $t$ of $U$ satisfying 
\begin{enumerate}
\item[(K1)] $s_a t_u=t_{a\cdot u}s_{a|_u}$; and
\item[(K2)] $s_a^*t_u=t_zs_{a|_z}^*$, where $z\in U$ is the unique element satisfying $a\cdot z=u$.
\end{enumerate}
Then there exists an isomorphism $\pi:C^*(U\bowtie A)\to \AA$ such that $\pi(v_{(u,a)})=t_us_a$ and $\pi(e_{(u,a)})=t_us_a s_a^*t_u^*$.
\end{thm}

\begin{proof}
Lemma~\ref{Davs delight} implies that $U\bowtie A$ is right LCM. We first find a family of isometries and projections in $\AA$ satisfying (L1)--(L4). Define $E_\varnothing:=0$, and for each $(u,a)\in U\bowtie A$ define
\[
V_{(u,a)}:=t_us_a\quad\text{and}\quad E_{(u,a)}:=V_{(u,a)}V_{(u,a)}^*=t_us_a s_a^*t_u^*.
\]
Then we use (K1) to get (L1):
\[
V_{(u,a)}V_{(w,b)}=t_us_a t_ws_b=t_ut_{a\cdot w}s_{a|_w}s_b=t_{u(a\cdot w)}s_{a|_wb}=V_{(u(a\cdot w),a|_wb)}=V_{(u,a)(w,b)}.
\]
It follows that
\[
V_{(u,a)}E_{(w,b)}V_{(u,a)}^*=V_{(u,a)}V_{(w,b)}V_{(w,b)}^*V_{(u,a)}^*=V_{(u,a)(w,b)}V_{(u,a)(w,b)}^*=E_{(u,a)(w,b)},
\]
which is (L2). We have $E_{(e_U,e_A)}=t_{e_U} s_{e_A}=1$, and $E_\varnothing=0$ by definition. So (L3) holds.


To prove (L4) first note that 
\[
E_{(u,a)}E_{(w,b)}=t_us_a s_a^* t_u^* t_w s_b s_b^* t_w^*.
\]
Suppose $(u,a)U\bowtie A\cap (w,b)U\bowtie A=\varnothing$. We know from Remark~\ref{rem: the lcm formula} that this means $uU\cap wU=\varnothing$. Since $t$ is covariant, we have $t_u^*t_w=0$, and hence $E_{(u,a)}E_{(w,b)}=0$.

Now suppose that $(u,a)U\bowtie A\cap (w,b)U\bowtie A=(z,c)U\bowtie A$. Then $uU\cap wU\not=\varnothing$. Let $u',w'\in U$ satisfy $uU\cap wU=uu'U=ww'U$ and $uu' = ww'$. Let $x,y\in U$ be the unique elements satisfying $a\cdot x=u'$ and $b\cdot y=w'$. Using the covariance of $t$ and condition (K2) we have
\begin{align*}
E_{(u,a)}E_{(w,b)}=t_us_a s_a^* t_u^* t_w s_b s_b^* t_w^* &= t_us_a s_a^* t_{u'} t_{w'}^* s_b s_b^* t_w^*\\
&= t_us_a t_xs_{a|_x}^*s_{b|_y}t_y^*s_b^* t_w^*.
\end{align*}
If $a|_x=b|_yb'$ for some $b'\in A$, then $(z,c)=(uu',a|_x)$. We can use (K1) to continue the calculation to get
\[
E_{(u,a)}E_{(w,b)}=t_ut_{a\cdot x}s_{a|_x}s_{b'}^*s_{b|_y}^*t_{b\cdot y}^*t_w^*=t_{u(a\cdot x)}s_{a|_x}s_{a|_x}^*t_{u(a\cdot x)}^*=t_zs_gs_g^*t_z^*=E_{(z,c)}.
\]
A similar argument gives $E_{(u,a)}E_{(w,b)}=E_{(z,c)}$ when $b|_y=a|_xa'$ for some $a'\in A$. So (L4) holds. It now follows from the universal property of $C^*(U\bowtie A)$ that there exists a homomorphism $\pi:C^*(U\bowtie A)\to \AA$ such that $\pi(v_{(u,a)})=t_us_a$ and $\pi(e_{(u,a)})=t_us_a s_a^*t_u^*$. 

To prove that $\pi$ is an isomorphism, we will find its inverse by constructing an isometric representation $S$ of $A$ in $C^*(U\bowtie A)$ and a covariant representation $T$ of $U$ in $C^*(U\bowtie A)$  satisfying (K1) and (K2). For each $u\in U$ and $a\in A$ let
\[
T_u:=v_{(u,e_A)}\quad\text{and}\quad S_a:=v_{(e_U,a)}.
\]
The fact that $T:u\mapsto T_u$ and $S:a\mapsto S_a$ are representations follows from the calculations
\[
T_uT_w=v_{(u,e_A)}v_{(w,e_A)}=v_{(u(e_A\cdot w),e_A|_we_A)}=v_{(uw,e_A)}=T_{uw}
\]
and
\[
S_a S_b=v_{(e_U,a)}v_{(e_U,b)}=v_{(e_U(a\cdot e_U),a|_{e_U} b)}=v_{(e_U,ab)}=S_{ab}.
\]
We also know that $T$ and $S$ are isometric because $v$ is isometric. To see that $T$ is covariant, first observe that \eqref{eq: Li range projections} implies that
\[
T_u^*T_w=v_{(u,e_A)}^*v_{(w,e_A)}=v_{(u,e_A)}^*(v_{(u,e_A)}v_{(u,e_A)}^*v_{(w,e_A)}v_{(w,e_A)}^*)v_{(w,e_A)}=v_{(u,e_A)}^*e_{(u,e_A)}e_{(w,e_A)}v_{(w,e_A)}.
\]
Now suppose $z$ is a right LCM of $u$ and $w$ and write $u'$ and $w'$ for elements of $U$ such that $uu'=ww'=z$. We know from Remark~\ref{rem: the lcm formula}(a) that $(z,e_A)$ is a right LCM of $(u,e_A)$ and $(w,e_A)$. Then (L4) gives
\begin{align*}
T_u^*T_w &= 
\begin{cases}
v_{(u,e_A)}^*e_{(z,e_A)}v_{(w,e_A)} & \text{if $uU\cap wU=zU$,}\\
0 & \text{if $uU\cap wU=\varnothing$}
\end{cases}\\
&= 
\begin{cases}
v_{(u,e_A)}^*v_{(z,e_A)}v_{(z,e_A)}^*v_{(w,e_A)} & \text{if $uU\cap wU=zU$,}\\
0 & \text{if $uU\cap wU=\varnothing$}
\end{cases}\\
&= 
\begin{cases}
v_{(u,e_A)}^*v_{(u,e_A)}v_{(u',e_A)}v_{(w',e_A)}^*v_{(w,e_A)}^*v_{(w,e_A)} & \text{if $uU\cap wU=zU$ and $uu'=ww'=z$,}\\
0 & \text{if $uU\cap wU=\varnothing$}
\end{cases}\\
&= 
\begin{cases}
v_{(u',e_A)}v_{(w',e_A)}^*& \text{if $uU\cap wU=zU$ and $uu'=ww'=z$,}\\
0 & \text{if $uU\cap wU=\varnothing$}
\end{cases}\\
&=
\begin{cases}
T_{u'}T_{w'}^*& \text{if $uU\cap wU=zU$ and $uu'=ww'=z$,}\\
0 & \text{if $uU\cap wU=\varnothing$.}
\end{cases}
\end{align*}
Hence $T$ is covariant. 

We need to show that (K1) and (K2) are satisfied. We have
\[
S_a T_u=v_{(e_U,a)}v_{(u,e_A)}=v_{(a\cdot u,a|_u)}=v_{(a\cdot u,e_A)}v_{(e_U,a|_u)}=T_{a\cdot u}S_{a|_u},
\]
which is (K1). For (K2) first recall from Remark~\ref{rem: the lcm formula}(b) that $(a\cdot z,a|_z)$ is a right LCM of $(e_U,a)$ and $(u,e_A)$, where $z$ is the unique element of $U$ with $a\cdot z=u$. Condition (L4) applied to these elements then becomes $e_{(e_U,a)}e_{(u,e_A)}=e_{(a\cdot z,a|_z)}$. Hence
\begin{align*}
S_a^*T_u&=v_{(e_U,a)}^*v_{(u,e_A)}\\
& =v_{(e_U,a)}^*v_{(e_U,a)}v_{(e_U,a)}^*v_{(u,e_A)}v_{(u,e_A)}^*v_{(u,e_A)}\\
&= v_{(e_U,a)}^*e_{(e_U,a)}e_{(u,e_A)}v_{(u,e_A)}\\
&= v_{(e_U,a)}^*e_{(a\cdot z,a|_z)}v_{(u,e_A)}\\
&= v_{(e_U,a)}^*v_{(a\cdot z,a|_z)}v_{(a\cdot z,a|_z)}^*v_{(u,e_A)}\\
&=v_{(e_U,a)}^*v_{(e_U,a)}v_{(z,e_A)}v_{(e_U,a|_z)}^*v_{(u,e_A)}^*v_{(u,e_A)}\\
&=v_{(z,e_A)}v_{(e_U,a|_z)}^*\\
&=T_zS_{a|_z},
\end{align*}
which is (K2). 

The universal property of $\AA$ gives a homomorphism $\phi:\AA\to C^*(U\bowtie A)$ with $\phi(t_u)=v_{(u,e_A)}$ and $\phi(s_a)=v_{(e_U,a)}$. We now check that $\pi$ and $\phi$ are inverses of each other using the generators:
\[
\pi\circ\phi(t_u)=\pi(v_{(u,e_A)})=t_us_{e_A}=t_u\quad\text{and}\quad \pi\circ(\phi(s_a))=\pi(v_{(e_U,a)})=t_{e_U} s_a=s_a,
\]
and
\begin{align*}
\phi\circ\pi(v_{(u,a)})=\phi(t_us_a)=v_{(u,e_A)}v_{(e_U,a)}=v_{(u,a)}\quad\text{and}\quad \phi\circ\pi(e_{(u,a)})&=\phi\circ\pi(v_{(u,a)}v_{(u,a)}^*)\\
&=v_{(u,a)}v_{(u,a)}^*\\
&= e_{(u,a)}.
\end{align*}
So $\pi:C^*(U\bowtie A)\to\AA$ is the desired isomorphism.
\end{proof}

\begin{remark}\label{qlog_ex}
The $C^*$-algebras associated with \zsps~generalise Nica's $C^*$-algebras of quasi-lattice ordered groups $(G,P)$. Recall from Section~\ref{subset: semigroup C*-algebras} that $C^*(G,P)$ is universal for representations $V$ of $P$ satisfying Equation~\eqref{eq: nica cov}. Consider the \zsp~$P \bowtie \{e\}$, where $e\cdot p=p$ and $e|_p=e$. The semigroups $P$ and $\{e\}$ satisfy the hypotheses of Lemma~\ref{Davs delight}. Conditions (K1) and (K2) from Theorem~\ref{thm: main c star theorem} are satisfied by definition, and so $C^*(P\bowtie\{e\})$ is the universal $C^*$-algebra generated by a covariant representation of $P$. Covariance, as in Definition~\ref{def: covariance}, is precisely Equation~\eqref{eq: nica cov} when $(G,P)$ is quasi-lattice ordered. So $C^*(P\bowtie\{e\})\cong C^*(G,P)$.
\end{remark}

\section{The boundary quotient}\label{sec: the boundary quotient}

In this section we introduce a quotient $\QQ(P)$ of Li's $C^*(P)$ for a right LCM semigroup $P$. Following the terminology of \cite{sy}, we say a subset $F\subseteq P$ is a {\em foundation set} if it is finite and for each $p\in P$ there exists $q\in F$ with $pP\cap qP\not=\varnothing$. When $(G,P)$ is quasi-lattice ordered, the collection of foundation sets in $P$ is described by Crisp and Laca in \cite[Definition 3.4]{cl}.

\begin{definition}\label{def: the boundary quotient}
Let $\QQ(P)$ be the universal $C^*$-algebra generated by isometries $\{v_p:p\in P\}$ and projections $\{e_{pP}:p\in P\}$ satisfying relations (1)--(4) of Definition~\ref{def:semigroupalgebra}, and 
\[
\prod_{p\in F}(1-e_{pP})=0\quad\text{for all foundation sets $F\subset P$}.
\]
We call $\QQ(P)$ the {\em boundary quotient} of $C^*(P)$.
\end{definition}

We now give an alternative presentation for the boundary quotient $\QQ(U\bowtie A)$.

\begin{thm}\label{thm: the boundary quotient for the bowtie}
Suppose $U$ and $A$ are semigroups with maps $(a,u) \mapsto a \cdot u$ and $(a,u) \mapsto a|_u$ satisfying (B1)--(B8) of Definition~\ref{def: the external bowtie}. Moreover, suppose $U$ is a right LCM semigroup, $A$ is left cancellative with $\Jj(A)$ totally ordered by inclusion, and for each $a\in A$, the map $u \mapsto a\cdot u$ is bijective. Then $\QQ(U\bowtie A)$ is the universal $C^*$-algebra generated by an isometric representation $s$ of $A$ and a covariant representation $t$ of $U$ satisfying (K1), (K2) and 
\begin{enumerate}
\item[(Q1)] $s_as_a^*=1$ for all $a\in A$; and
\item[(Q2)] $\displaystyle{\prod_{u\in F}(1-t_ut_u^*)=0}$ for all foundation sets $F\subseteq U$.
\end{enumerate}
\end{thm}

To prove this result we need the following lemma about foundation sets.

\begin{lemma}\label{lem: foundation sets}
Suppose $U$ and $A$ are semigroups with maps $(a,u) \mapsto a \cdot u$ and $(a,u) \mapsto a|_u$ satisfying (B1)--(B8) of Definition~\ref{def: the external bowtie}. Moreover, suppose $U$ is a right LCM semigroup, $A$ is left cancellative with $\Jj(A)$ totally ordered by inclusion, and for each $a\in A$, the map $u \mapsto a\cdot u$ is bijective.
\begin{enumerate}
\item[(a)] For every $a\in A$ the singleton set $\{(e_U,a)\}$ is a foundation set in $U\bowtie A$.
\item[(b)] For every foundation set $F\subseteq U$ the set $\{(u,e_A):u\in F\}$ is a foundation set in $U\bowtie A$.
\item[(c)] For every foundation set $G$ in $U\bowtie A$ the set $\{u\in U: (u,a)\in G\text{ for some }a\in A\}$ is a foundation set in $U$.
\end{enumerate}
\end{lemma}

\begin{proof}
To prove (a), fix $a\in A$ and consider an arbitrary $(u,b)\in U\bowtie A$. Let $z$ be the unique element of $U$ with $a\cdot z=u$. Let $a',b'\in A$ with $a|_za'=bb'=:c$. (Since $\JJ(A)$ is totally ordered we know that at least one of $a'$ or $b'$ is $e_A$.) Then 
\[
(u,c)=(e_U,a)(z,a')\quad\text{and}\quad (u,c)=(u,b)(e_U,b').
\]
Hence $(e_U,a)U\bowtie A\cap (u,b)U\bowtie A\not=\varnothing$, and so (a) holds.

For (b), let $F$ be a foundation set in $U$ and let $(v,b)\in U\bowtie A$. Then there exists $u\in F$ with $uU\cap vU\not=\varnothing$. Let $u',v',w\in U$ with $w=uu'=vv'$. Let $x$ be the unique element of $U$ with $b\cdot x=v'$. Then
\[
(w, b|_x)=(u,e_A)(u', b|_x) \quad\text{and}\quad (w, b|_x) = (v,b)(x,e_A).
\]
Hence $(u,e_A)U\bowtie A\cap (v,b)U\bowtie A\not=\varnothing$, and so $\{(u,e_A):u\in F\}$ is a foundation set in $U\bowtie A$.

For (c), let $G$ be a foundation set in $U\bowtie A$ and let $v\in U$. Then there exists $(u,a)\in G$ with $(u,a)U\bowtie A\cap (v,e_A)U\bowtie A\not=\varnothing$. But this means there exists $u',v'$ with $uu'=vv'$, and hence $uU\cap vU\not=\varnothing$. So $\{u\in U: (u,a)\in G\text{ for some }a\in A\}$ is a foundation set in $U$. 
\end{proof}

\begin{proof}[Proof of Theorem~\ref{thm: the boundary quotient for the bowtie}]
Under the presentation of $C^*(U\bowtie A)$ established in Theorem~\ref{thm: main c star theorem}, products $\prod_{(u,a)\in G}(1-e_{(u,a)})$ over foundation sets $G\subset U\bowtie A$ correspond to $ \prod_{(u,a)\in G} (1 - t_u s_a s_a^* t_u^*)$. So it suffices to show that conditions (Q1) and (Q2) are equivalent to the condition
\[
 \prod_{(u,a)\in G} (1 - t_u s_a s_a^* t_u^*) = 0
\]
for all foundation sets $G \subset U\bowtie A$. To see this, first suppose that (Q1) and (Q2) hold and fix a foundation set $G\subset U\bowtie A$. Then $s_as_a^*=1$ for each $a\in A$ by (Q1), and hence
\[
 \prod_{(u,a)\in G} (1 - t_u s_a s_a^* t_u^*) = \prod_{(u,a)\in G} (1 - t_u t_u^*).
 \]
 Since $\{u\in U: (u,a)\in G\text{ for some }a\in A\}$ is a foundation set by Lemma \ref{lem: foundation sets}, we know from (Q2) that the above product is zero. Conversely, suppose
\[
 \prod_{(u,a)\in G} (1 - t_u s_a s_a^* t_u^*) = 0
\]
for all foundation sets $G\subset U\bowtie A$. Fix $F\subset U$ a foundation set. Then by Lemma \ref{lem: foundation sets}, the set $F'=\{(u,e_A):u\in F\} \subset U\bowtie A$ is a foundation set, and hence
\[
\prod_{u\in F} (1-t_u t_u^*) = \prod_{u\in F'} (1-t_u s_{e_A} s_{e_A}^* t_u^*) = 0.
\]
Likewise, for any $a\in A$ Lemma \ref{lem: foundation sets} implies the singleton set $\{(e_U,a)\} \subset U\bowtie A$ is a foundation set. So
\[
 1-s_as_a^*=1-t_{e_U} s_a s_a^* t_{e_U}^* = 0,
\]
and hence $s_a s_a^* = 1$ as required.
\end{proof}

\begin{remark}\label{rem: all the boundary quotients}
We can see from the presentation of $\QQ(U\bowtie A)$ that we potentially have two other quotients of $C^*(U\bowtie A)$. We denote by $C^*_A(U\bowtie A)$ the quotient obtained from adding relation (Q1) to the relations of $C^*(P)$, and by $C^*_U(U\bowtie A)$ the quotient obtained from adding relation (Q2). These quotients are interesting in their own right (in the case of $\nxn$ these quotients have been studied in \cite{BaHLR}), and we will discuss them further throughout the next section.
\end{remark}

\begin{remark}\label{rem: general boundary quotient}
The definition of the boundary quotient from Definition~\ref{def: the boundary quotient} has a natural generalisation to arbitrary discrete left cancellative semigroups. For such a $P$ we say $F\subset\JJ(P)$ is a {\em foundation set} if $F$ is finite, and for each $Y\in\JJ(P)$ there exists $X\in F$ with $X\cap Y\not=\varnothing$. We define $\QQ(P)$ to be the universal $C^*$-algebra generated by isometries $\{v_p:p\in P\}$ and projections $\{e_{X}:X\in \JJ(P)\}$ satisfying relations (1)--(4) of Definition~\ref{def:semigroupalgebra}, and 
\[
\prod_{X\in F}(1-e_{X})=0\quad\text{for all foundation sets $F\subset P$}.
\]

\end{remark}

\section{Examples}\label{sec: examples}

\subsection{Baumslag-Solitar groups}\label{subsec: baumslag-solitar groups C*} 

Consider the Baumslag-Solitar group $BS(c,d)$, for positive integers $c$ and $d$. Recall from Section~\ref{subsec: baumslag-solitar groups} that $\BS(c,d)^+\cong U\bowtie A$, where $U\cong \F_d^+$, $A\cong \N$, and the action and restriction maps are given in \eqref{eq: BS action} and \eqref{eq: BS restriction}.

\begin{prop}\label{prop: BS BQ}
The boundary quotient $\QQ(\BS(c,d)^+)$ is the universal $C^*$-algebra generated by a unitary $s$ and isometries $t_1,\dots, t_d$ satisfying
\begin{enumerate}
\item[(1)] $\sum_{i=1}^dt_it_i^*=1$;
\item[(2)] $st_i=t_{i+1}$ for $1\le i <d$; and
\item[(3)] $st_d=t_1s^c$.
\end{enumerate} 
Moreover, $\QQ(\BS(c,d)^+)$ is isomorphic to the category of paths algebra $C^*(\Lambda)$ from \cite{Sp}.
\end{prop}

\begin{proof}
First note that $U$ and $A$ satisfy the hypotheses of Theorem \ref{thm: the boundary quotient for the bowtie}, so we can use the given presentation of $\QQ(\BS(c,d)^+)$. That $s$ is unitary follows from (Q1). Since $U$ is $\F_d^+$, it suffices to only consider the foundation set consisting of generators of $\F_d^+$ in (Q2). Hence we have (1). Relations (2) and (3) are (K1). Relation (K2) follows from (2) and (3) and that $s$ is unitary. It follows immediately from the generators and relations presented in \cite[Theorem~3.23]{Sp} that $\QQ(\BS(c,d)^+)$ is isomorphic to $C^*(\Lambda)$.
\end{proof} 

We are now able to use \cite[Remark~3.25]{Sp} to link $\QQ(\BS(c,d)^+)$ with topological graph algebras.

\begin{cor}
The $C^*$-algebra $\QQ(\BS(c,d)^+)$ is isomorphic to the topological graph algebra $\OO(E_{d,c})$ from \cite{Kat}.
\end{cor}

From the discussion in \cite[Example~A.6]{Kat}, we have the following corollary.
 
\begin{cor}\label{cor: BS and Kirchberg}
If $c\not\in d\Z$, then $\QQ(\BS(c,d)^+)$ is a Kirchberg algebra.
\end{cor}

\subsection{The semigroup $\nxn$}\label{subsec: n by n times C*}

Recall from \cite[Theorem~4.1]{lr} that the Toeplitz algebra $\TT(\nxn)$ is the universal $C^*$-algebra generated by an isometry $s$  and isometries $v_p$ for each prime $p$ satisfying
\begin{enumerate}
\item[(T1)] $v_ps=s^pv_p$; 
\item[(T2)] $v_pv_q=v_qv_p$;
\item[(T3)] $v_p^*v_q=v_qv_p^*$ for $p\not=q$;
\item[(T4)] $s^*v_p=s^{p-1}v_ps^*$; and
\item[(T5)] $v_p^*s^kv_p=0$ for all $1\le k<p$.
\end{enumerate}
The boundary quotient (in the sense of \cite{cl}) of $\TT(\nxn)$ is Cuntz's $\QQ_\N$ from \cite{c2}, and corresponds to adding the following relations:
\begin{enumerate}
\item[(Q5)] $\sum_{k=0}^{p-1}(s^kv_p)(s^kv_p)^*=1$ for all primes $p$; and
\item[(Q6)] $ss^*=1$.
\end{enumerate}

We saw in Section~\ref{subsec: n by n times} that $\nxn$ is the internal {\zsp} $U\bowtie A$, where
\[
U = \{(r,x) : x \in \N^\times, 0\le r \le x-1\}\quad\text{and}\quad A = \{(m,1) : m \in \N\}.
\]
Moreover, $U$ and $A$ satisfy the hypotheses of Theorem~\ref{thm: main c star theorem} and Theorem~\ref{thm: the boundary quotient for the bowtie}, so we can apply these theorems to $C^*(\nxn)$ and $\QQ(\nxn)$, respectively.

\begin{prop}\label{prop: main C* result for nxn}
There is an isomorphism $\phi:\TT(\nxn)\to C^*(\nxn)$ satisfying $\phi(s)=s_{(1,1)}$ and $\phi(v_p)=t_{(0,p)}$ for all primes $p$. The isomorphism $\phi$ descends to an isomorphism of $\QQ_\N$ onto $\QQ(\nxn)$. 
\end{prop}

\begin{proof}
The Toeplitz algebra $\TT(\nxn)$ can be viewed as the universal $C^*$-algebra generated by a Nica covariant representation $V$ of $\nxn$. This description coincides with the presentation (T1)--(T5) via $s\mapsto V_{(1,1)}$ and each $v_p\mapsto V_{(0,p)}$ \cite[Page 652]{lr}. The isomorphism $\TT(\nxn)\to C^*(\nxn)$ from Li's argument in \cite[Section~2.4]{Li2012} sends $V_{(1,1)}\mapsto v_{(1,1)}$ and each $V_{(0,p)}\mapsto v_{(0,p)}$. The isomorphism of Theorem~\ref{thm: main c star theorem} sends $v_{(1,1)}\mapsto s_{(1,1)}$ and each $v_{(0,p)}\mapsto t_{(0,p)}$. We define $\phi:\TT(\nxn)\to C^*(\nxn)$ to be the composition of these isomorphisms.

We now claim that if $s_{(1,1)}$ and $\{t_{(0,p)}:p\text{ prime}\}$ satisfy (Q1) and (Q2) of Theorem~\ref{thm: the boundary quotient for the bowtie}, then they satisfy (Q5) and (Q6). 

For $p$ a prime and $0\le k\le p-1$ we have $(k,1)\cdot (0,p)=(k,p)$ and $(k,1)|_{(0,p)}=(0,1)$. Using (K1) of Theorem~\ref{thm: main c star theorem} we then get 
\begin{equation}\label{eq: 1 for C* nxn}
s_{(k,1)}t_{(0,p)}=t_{(k,p)}.
\end{equation}
Since the sets $\{(k,p)\nxn:0\le k\le p-1\}$ are mutually disjoint and $t$ is covariant, we have $t_{(j,p)}t_{(j,p)}^*t_{(k,p)}t_{(k,p)}^*=0$ for $0\le j,k\le p-1$ with $j\not= k$. Hence
\[
\prod_{k=0}^{p-1}(1-t_{(k,p)}t_{(k,p)}^*)=1-\sum_{k=0}^{p-1}t_{(k,p)}t_{(k,p)}^*.
\]
Since $\{(0,p),\dots,(p-1,p)\}$ is a foundation set in $U$, it follows from (Q2) that
\begin{equation}\label{eq: 2 for C* nxn}
\sum_{k=0}^{p-1}t_{(k,p)}t_{(k,p)}^*=1.
\end{equation}
It now follows from Equations~\eqref{eq: 1 for C* nxn} and \eqref{eq: 2 for C* nxn} that
\[
\sum_{k=0}^{p-1}s_{(1,1)}^kt_{(0,p)}(s_{(1,1)}^kt_{(0,p)})^*=\sum_{k=0}^{p-1}(s_{(k,1)}t_{(0,p)})(s_{(k,1)}t_{(0,p)})^*=\sum_{k=0}^{p-1}t_{(k,p)}t_{(k,p)}^*=1,
\]
and hence (Q5) is satisfied. We know from (Q2) that $s_{(1,1)}s_{(1,1)}^*=1$, and hence (Q6) is satisfied. So $\phi$ descends to a homomorphism $\phi:\QQ_\N\to \QQ(\nxn)$, which is injective because $\QQ_\N$ is simple \cite[Theorem 3.4]{c2}. For each $(r,x)\in U$ and $(m,1)\in A$ we have 
\[
\phi(s^m)=s_{(m,1)}\quad\text{and}\quad\phi(s^rv_x)=t_{(r,x)},
\]
and so each generator of $\QQ(\nxn)$ is in the range of $\phi$. Hence $\phi:\QQ_\N\to\QQ(\nxn)$ is surjective, and so is an isomorphism.
\end{proof}

\begin{remark}\label{rem: other quotients for n by n times}
The multiplicative and boundary quotients of $\TT(\nxn)$ are studied in \cite{BaHLR}. The multiplicative boundary quotient $\TT_{\mult}(\nxn)$ corresponds to adding relation (Q5) to the presentation of $\TT(\nxn)$, and the additive boundary quotient $\TT_{\add}(\nxn)$ corresponds to adding relation (Q6). It follows from Proposition~\ref{prop: main C* result for nxn} that $\TT_{\mult}(\nxn)\cong C^*_U(U\bowtie A)$ and $\TT_{\add}(\nxn)\cong C^*_A(U\bowtie A)$.
\end{remark}

\subsection{The semigroup $\zxz$}\label{subsec: zxz C*}

In \cite{c2} Cuntz also introduced the $C^*$-algebra $\QQ_\Z$, which instigated work on the $C^*$-algebras of more general integral domains \cite{cLi}.  
Recall that $\QQ_\Z$ can be viewed as the universal $C^*$-algebra generated by a unitary $s$ and isometries $\{v_a:a\in\Z^\times\}$ satisfying
\begin{enumerate}
\item[(i)] $v_av_b=v_{ab}$ for all $a,b\in\Z^\times$;
\item[(ii)] $v_as=s^av_a$ and $v_as^*={s^*}^av_a$ for all $a\in\Z^\times$; and 
\item[(iii)] $\sum_{j=0}^{|a|-1}s^jv_av_a^*{s^*}^{j}=1$ for all $a\in\Z^\times$.
\end{enumerate}

Recall from Section~\ref{subsec: n by n times} that $\zxz$ is the internal {\zsp} $U\bowtie A$, where
\[
U = \{(r,x) : x >0, 0\le r \le x-1\}\quad\text{and}\quad A = \Z\times\{1,-1\}.
\]
Moreover, $U$ and $A$ satisfy the hypotheses of Theorem~\ref{thm: the boundary quotient for the bowtie}, so we can use the presentation of $\QQ(\zxz)$ from Theorem~\ref{thm: the boundary quotient for the bowtie}. 

\begin{prop}\label{prop: main C* result for zxz}
There is an isomorphism $\phi:\QQ_\Z\to \QQ(\zxz)$ satisfying $\phi(s)=s_{(1,1)}$ and 
\[
\phi(v_a)=s_{(0,a/|a|)}t_{(0,|a|)}\text{ for all $a\in\Z^\times$}.
\]
\end{prop}

\begin{proof}
Let $S$ denote the unitary $s_{(1,1)}$ and let $V_a$ denote the isometry $s_{(0,a/|a|)}t_{(0,|a|)}$ for each $a\in \Z^\times$. We claim that (i)--(iii) are satisfied. First note that for each $j\in\{1,-1\}$ and $a\in\Z^\times$ we have $(0,j)\cdot(0,|a|)=(0,|a|)$ and $(0,j)|_{(0,|a|)}=(0,j)$. (Of course, $(0,1)$ is the identity in $\zxz$, so these identities trivially hold when $j=1$.) Hence by (K1) we have
\[
s_{(0,j)}t_{(0,|a|)}=t_{(0,j)\cdot(0,|a|)}s_{(0,j)|_{(0,|a|)}}=t_{(0,|a|)}s_{(0,j)},
\]
for all $a\in\Z^\times$. Now
\[
V_aV_b=s_{(0,a/|a|)}t_{(0,|a|)}s_{(0,b/|b|)}t_{(0,|b|)}=s_{(0,a/|a|)}s_{(0,b/|b|)}t_{(0,|a|)}t_{(0,|b|)}=s_{(0,ab/|ab|)}t_{(0,|ab|)}=V_{ab},
\]
which is (i). For the first part of (ii), first note that $(|a|,1)\cdot(0,|a|)=(0,|a|)$ and $(|a|,1)|_{(0,|a|)}=(1,1)$ for all $a\in\Z^\times$. Hence by (K1) we have 
\begin{align*}
V_aS=s_{(0,a/|a|)}t_{(0,|a|)}s_{(1,1)} &= s_{(0,a/|a|)}t_{(|a|,1)\cdot(0,|a|)}s_{(|a|,1)|_{(0,|a|)}} \\
&= s_{(0,a/|a|)}s_{(|a|,1)}t_{(0,|a|)}\\
&= s_{(a,a/|a|)}t_{(0,|a|)}\\
&= s_{(1,1)}^as_{(0,a/|a|)}t_{(0,|a|)}\\
&=S^aV_a.
\end{align*} 
Similarly we have $(-|a|,1)\cdot(0,|a|)=(0,|a|)$ $(-|a|,1)|_{(0,|a|)}=(-1,1)$ for all $a\in\Z^\times$, and hence
\begin{align*}
V_aS^*=s_{(0,a/|a|)}t_{(0,|a|)}s_{(-1,1)} &= s_{(0,a/|a|)}t_{(-|a|,1)\cdot(0,|a|)}s_{(-|a|,1)|_{(0,|a|)}} \\
&= s_{(0,a/|a|)}s_{(-|a|,1)}t_{(0,|a|)}\\
&= s_{(-a,a/|a|)}t_{(0,|a|)}\\
&= s_{(-1,1)}^as_{(0,a/|a|)}t_{(0,|a|)}\\
&={S^*}^aV_a.
\end{align*} 
So (ii) holds. For (iii), we first calculate 
\begin{align*}
\sum_{j=0}^{|a|-1}S^jV_aV_a^*{S^*}^j &= \sum_{j=0}^{|a|-1}s_{(1,1)}^js_{(0,a/|a|)}t_{(0,|a|)}t_{(0,|a|)}^*s_{(0,a/|a|)}^*{s_{(1,1)}^*}^j\\
&=  \sum_{j=0}^{|a|-1}s_{(j,a/|a|)}t_{(0,|a|)}t_{(0,|a|)}^*s_{(j,a/|a|)}^*.
\end{align*}
Now, (K1) gives
\[
s_{(j,a/|a|)}t_{(0,|a|)}=t_{(j,a/|a|)\cdot(0,|a|)}s_{(j,a/|a|)|_{(0,|a|)}}=t_{(j,|a|)}s_{(0,a/|a|)}.
\]
Hence 
\[
\sum_{j=0}^{|a|-1}S^jV_aV_a^*{S^*}^j =\sum_{j=0}^{|a|-1}t_{(j,|a|)}s_{(0,a/|a|)}s_{(0,a/|a|)}^*t_{(j,|a|)}^*=\sum_{j=0}^{|a|-1}t_{(j,|a|)}t_{(j,|a|)}^*,
\]
where the last equality holds because $s_{(0,a/|a|)}$ is a unitary. Since $\{(0,a),\dots, (a-1,a)\}$ is a foundation set, and the corresponding principal ideals are mutually disjoint, condition (Q2) gives
\[
\sum_{j=0}^{|a|-1}S^jV_aV_a^*{S^*}^j =\sum_{j=0}^{|a|-1}t_{(j,|a|)}t_{(j,|a|)}^*=1.
\]
So (iii) holds. We then get a homomorphism $\phi:\QQ_\Z\to \QQ(\zxz)$ satisfying $\phi(s)=s_{(1,1)}$ and $\phi(v_a)=s_{(0,a/|a|)}t_{(0,|a|)}$ for all $a\in\Z^\times$. Since $\QQ_\Z$ is simple, we know that $\phi$ is injective. For each $(r,x)\in U$ and $(m,j)\in A$ ($j\in\{1,-1\}$), we have
\[
\phi(s^rv_x)=t_{(r,x)}\quad\text{and}\quad \phi(s^m)=s_{(m,1)}.
\]
Since each generator is in the range of $\phi$, it follows that $\phi$ is the desired isomorphism.
\end{proof}

\subsection{Self-similar actions}\label{subsec: self-similar actions C*}

Let $(G,X)$ be a self-similar action as described in Section~\ref{subsec: self-similar actions}. Recall from \cite[Proposition 3.2]{lrrw} that the Toeplitz algebra $\TT(G,X)$ is the universal $C^*$-algebra generated by a Toeplitz-Cuntz family of isometries $\{v_x \mid x \in X\}$ and a unitary representation $u$ of $G$ satisfying
\begin{equation}\label{SS equation}
u_g v_x = v_{g\cdot x} u_{g|_{x}},\quad\text{for all $g\in G$ and $x\in X$}.
\end{equation}
Recall from \cite[Definition 3.1]{N2009} (see also \cite[Corollary 3.5]{lrrw}) that the Cuntz-Pimsner algebra $\OO(G,X)$ is the quotient of $\TT(G,X)$ by the ideal $I$ generated by $1-\sum_{x\in X}v_xv_x^*$.

In Section~\ref{subsec: self-similar actions} we saw that each self-similar action $(G,X)$ gives rise to a \ZS product $X^*\bowtie G$. We now show that $C^*(X^* \bowtie G)$ is isomorphic to $\TT(G,X)$, and that the boundary quotient $\QQ(X^* \bowtie G)$ is isomorphic to $\OO(G,X)$.

\begin{thm}
Let $(G,X)$ be a self-similar group. There is an isomorphism $\phi:\TT(G,X) \to C^*(X^* \bowtie G)$ such that $\phi(u_g)=s_g$ and $\phi(v_x)=t_x$. Moreover, $\phi$ descends to an isomorphism of $\OO(G,X)$ onto $\QQ(X^* \bowtie G)$.
\end{thm}

\begin{proof}
We know that $s$ is a unitary representation of $G$ in $C^*(X^*\bowtie G)$. The covariance of $t$ is equivalent to insisting that $t_x^*t_y=\delta_{x,y}$; that is, $\{t_x:x\in X\}$ is a Toeplitz-Cuntz family. Condition (K1) is equivalent to insisting that $s_g t_x = t_{g\cdot x} s_{g|_{x}}$ for all $g\in G$ and letters $x\in X$; to replace $x$ in this equation with a word $w\in X^*$ we use (B5) and (B6) of Definition~\ref{def: the external bowtie}. Condition (K2) comes for free from (K1): for $w,z\in X^*$ and $g\in G$ with $g\cdot z=w$ we have $s_gt_z=t_{g\cdot z}s_{g|_z}=t_ws_{g|_z}$, and then
\[
t_ws_{g|_z}=s_gt_z \Longleftrightarrow s_g^*t_ws_{g|_z}s_{g|_z}^*=s_g^*s_gt_zs_{g|_z}^* \Longleftrightarrow s_g^*t_w=t_zs_{g|_z}^*.
\]
The above arguments imply that $C^*(X^*\bowtie G)$ is the universal $C^*$-algebra generated by a Toeplitz-Cuntz family $\{t_x:x\in X\}$ and a unitary representation $s$ of $G$ satisfying $s_g t_x = t_{g\cdot x} s_{g|_{x}}$ for all $g\in G$ and $x\in X$. It is now evident that we have the desired isomorphism $\phi:\TT(G,X) \to C^*(X^* \bowtie G)$.

It remains to show that $\phi$ descends to an isomorphism $\phi:\OO(G,X) \to \QQ(X^* \bowtie G)$. Denote by $I$ the ideal in $\TT(G,X)$ generated by $1-\sum_{x\in X} v_xv_x^*$, and by $J$ the ideal in $C^*(X^*\bowtie G)$ generated by the set $\{\prod_{w\in F}(1-t_wt_w^*):F\text{ is a foundation set in }X^*\}$. So $\OO(G,X)$ is $\TT(G,X)/I$ and $\QQ(X^*\bowtie G)$ is $C^*(X^*\bowtie G)/J$. We need to show that $\phi(I)$, which is the ideal in $C^*(X^*\bowtie G)$ generated by $1-\sum_{x\in X}t_xt_x^*$, is equal to $J$.

Since $X$ is a foundation set in $X^*$, we have $\phi(I)\subset J$. To get the reverse containment we fix a foundation set $F$; it suffices to prove that $\prod_{w\in F}(1-t_wt_w^*)\in \phi(I)$. Denote
\[
N:=\max\{|w|:w\in F\}\quad\text{and}\quad F':=\bigcup_{w\in F}\{ww':w'\in X^{N-|w|}\}.
\]
We now claim that $F'=X^N$. For the sake of contradiction, suppose $F' \neq X^N$ and let $z \in X^N \setminus F'$. By definition of $F'$, it follows that there is no $w \in F$ such that $z=ww'$. This means $wX^* \cap zX^*= \varnothing$ for all $w\in F$, which contradicts that $F$ is a foundation set. So we must have $F'=X^N$. For each $w\in F$ and $ww'\in X^N$ we have $1-t_wt_w^*\le 1-t_{ww'}t_{ww'}^*$. Since the projections $\{1-t_zt_z^*: z\in X^N\}$ commute, we then have
\[
\prod_{w\in F}(1-t_wt_w^*)= \prod_{w\in F}(1-t_wt_w^*)\prod_{z\in X^N}(1-t_zt_z^*).
\]
The result will now follow if $\prod_{z\in X^N}(1-t_zt_z^*)=1-\sum_{z\in X^N}t_zt_z^*\in\phi(I)$. We show by induction that $1-\sum_{z\in X^n}t_zt_z^*\in\phi(I)$ for all $n\ge 1$. The result is true for $n=1$. Assume true for $n$. Then
\begin{align*}
1-\sum_{w\in X^{n+1}}t_wt_w^* &= 1-\sum_{x\in X}\sum_{z\in X^n}t_{xz}t_{xz}^*\\
&=1-\sum_{x\in X}t_x\big(\sum_{z\in X^n}t_zt_z^*\big)t_x^*\\
&= 1+\sum_{x\in X}t_x\big(\big(1-\sum_{z\in X^n}t_zt_z^*\big) -1\big)t_x^*\\
&= \big(1-\sum_{x\in X}t_xt_x^*\big) + \sum_{x\in X}t_x\big(1-\sum_{z\in X^n}t_zt_z^*\big)t_x^*\\
&\in \phi(I)
\end{align*}
So $\prod_{w\in F}(1-t_wt_w^*)\in\phi(I)$, and we have $J\subset \phi(I)$. Hence $\phi(I)=J$, and $\phi$ descends to an isomorphism of $\OO(G,X)$ onto $\QQ(X^*\bowtie G)$.
\end{proof}

\subsection{The binary adding machine}\label{subsec: the adding machine C star C*}

The 2-adic ring $C^*$-algebra of the integers $\QQ_2$ was introduced and studied in \cite{LL}. Recall that $\QQ_2$ is simple and purely infinite, and is the universal $C^*$-algebra generated by a unitary $u$ and an isometry $s_2$ satisfying 
\begin{enumerate}
\item[(I)] $s_2u=u^2s_2$; and
\item[(II)] $s_2s_2^*+us_2s_2^*u^*=1$.
\end{enumerate}

Consider the \ZS product $X^*\bowtie\N$ described in Section~\ref{subsec: the adding machine}, where $X$ is the alphabet $\{0,1\}$, and $\N=\langle e,\gamma\rangle$. It follows from our identification of $X^*\bowtie \N$ as $\BS(1,2)^+$ that $C^*(X^*\bowtie \N)$ is isomorphic to Nica's $C^*(\BS(1,2),\BS(1,2)^+)$. The quotient $C^*_\N(X^*\bowtie \N)$ (in the sense of Remark~\ref{rem: all the boundary quotients}) is isomorphic to the Cuntz-Pimsner algebra $\OO(\Z,X)$ described in Section \ref{subsec: self-similar actions C*}. We also have the following description of the boundary quotient:

\begin{prop}\label{prop: Q 2 C* result}
There is an isomorphism $\phi:\QQ_2\to\QQ(X^*\bowtie \N)$ such that $\phi(u)=s_\gamma$ and $\phi(s_2)=t_0$.
\end{prop}

\begin{proof}
We claim that
\[
U:=s_\gamma\in\QQ(X^*\bowtie\N)\quad\text{and}\quad S_2:=t_0\in\QQ(X^*\bowtie \N)
\]
satisfy relations (I) and (II). First note that $U$ is unitary because of (Q1). Recall from the formulae in Section~\ref{subsec: the adding machine} that $\gamma\cdot 0=1$, $\gamma|_0=e$, $\gamma\cdot 1=0$ and $\gamma|_1=\gamma$. It then follows from (K1) that
\[
US_2=s_\gamma t_0=t_{\gamma\cdot 0}s_{\gamma|_0}=t_1s_e=t_1.
\]
Hence
\[
U^2S_2=U(US_2)=s_\gamma t_1=t_{\gamma\cdot 1}s_{\gamma|_1}=t_0s_\gamma=S_2U,
\]
and so (I) is satisfied. For (II) first note that, since $\{0,1\}\subset X^*$ is a foundation set, we have from (Q2) that $(1- t_0t_0^*)(1-t_1t_1^*)=0$. Since $t_0t_0^*$ and $t_1t_1^*$ are orthogonal, their sum is one. Hence
\[
S_2S_2^*+US_2S_2^*U^*=t_0t_0^*+t_1t_1^*=1.
\]
So (II) is satisfied. The universal property of $\QQ_2$ gives a homomorphism $\phi:\QQ_2\to\QQ(X^*\bowtie\N)$ such that $\phi(u)=s_\gamma$ and $\phi(s_2)=t_0$. We know $\phi$ is injective because $\QQ_2$ is simple. Since the generators of $\QQ(X^*\bowtie \N)$ are all in the range of $\phi$, we know $\phi$ is surjective, and hence an isomorphism.  
\end{proof}

\subsection{Products of self-similar actions}\label{subsec: products of SSAs C*}

Consider a \ZS product $\F_\theta^+ \bowtie G$ coming from a product 
of two 
self-similar actions as constructed in Section 
\ref{subsec:products_SSAs}, 
where $\F_\theta^+$ is right LCM. In general, the $C^*$-algebra 
$C^*(\F_\theta^+\bowtie G)$ is the universal $C^*$-algebra generated 
by a 
Toeplitz-Cuntz-Krieger family $\{s_\lambda : \lambda \in 
\F_\theta^+\}$ and a 
unitary representation $u$ of $G$ satisfying
\begin{equation} \label{eq:productSSA C*}
 u_g s_\lambda = s_{g\cdot\lambda} u_{g|_\lambda}.
\end{equation}
for all $g\in G$ and $\lambda \in \F_\theta^+$.
Likewise, the quotient $\QQ(\F_\theta^+\bowtie G)$ is the universal 
$C^*$-algebra generated by a Cuntz-Krieger family $\{s_\lambda : 
\lambda \in 
\F_\theta^+\}$ and a unitary representation $u$ of $G$ satisfying 
(\ref{eq:productSSA C*}).

Recall from Section~\ref{exm:productaddingmachines} the semigroup 
$\F_\theta^+\bowtie\Z$ 
constructed from the product of two adding machines 
$(\Z,\{0,\dots,m-1\})$ and 
$(\Z,\{0,\dots,n-1\})$. If $m$ and $n$ are coprime, $\F_\theta^+$ is 
right LCM, and we 
can apply our theorems to $\F_\theta^+\bowtie\Z$. The $C^*$-algebra 
$C^*(\F_\theta^+)$ has been studied extensively in \cite{DY2009}; 
Corollary~3.2 of \cite{DY2009} implies that the $2$-graph 
$\F_\theta^+$ is 
aperiodic, and hence $C^*(\F_\theta^+)$ is simple by \cite[Theorem 
3.1]{RS2007}. It would be interesting, albeit outside the scope of 
this paper, to further 
understand $C^*(\F_\theta^+\bowtie\Z)$ from the point of view of 
these existing 
constructions.



\begin{thebibliography}{20}

\bibitem{Brin} M.G. Brin, \emph{On the Zappa-Sz\'{e}p product}, Comm. in Algebra {\bf 33} (2005), 393--424.

\bibitem{BaHLR} N. Brownlowe, A. an Huef, M. Laca and I. Raeburn, {\em Boundary quotients of the Toeplitz algebra of the affine semigroup over the natural numbers}, Ergodic Theory Dynam. Systems {\bf 32} (2012), 35--62.

\bibitem{Cob} L. A. Coburn, {\em The $C^*$-algebra generated by an isometry, I}, Bull. Amer. Math. Soc. {\bf 73} (1967), 7222--726.

\bibitem{cl} J. Crisp and M. Laca, {\em Boundary quotients and ideals of Toeplitz $C^*$-algebras of Artin groups}, J. Funct. Anal. {\bf 242} (2007), 127--156.

\bibitem{c2} J. Cuntz, {\em $C^*$-algebras associated with the $ax+b$-semigroup over $\N$}, $K$-Theory and Noncommutative Geometry (Valladolid, 2006), European Math. Soc., 2008, pages~201--215.

\bibitem{cdl} J. Cuntz, C. Deninger and M. Laca, {\em $C^*$-algebras of Toeplitz type associated with algebraic number fields},  Math. Ann. {\bf 355} (2013), 1383--1423.

\bibitem{cLi} J. Cuntz and X. Li, {\em The regular $C^*$-algebra of an integral domain}, Quanta of Maths, Clay Math. Proc., vol. 11, American Mathematical Society, Providence, RI, (2010), 149--170.

\bibitem{dpy} K.R. Davidson, S.C. Power and D. Yang, {\em Dilation theory for rank 2 graph algebras}, J. Operator Theory {\bf 63} (2010), 245--270.

\bibitem{DY2009} K.R. Davidson and D. Yang, \emph{Periodicity in rank $2$ graph algebras} Canad. J. Math. {\bf 61} (2009), 1239--1261.

\bibitem{d} R.G. Douglas, {\em On the $C^*$-algebra of a one-parameter semigroup of isometries}, Acta Math. {\bf 128} (1972), 143--151.

\bibitem{Gri80} R.I. Grigorchuk, {\em On Burnside's problem on periodic groups}, Func. Anal. Appl. {\bf 14} (1980), 41--43.

\bibitem{Kat} T. Katsura, \emph{A class of $C^*$-algebras generalizing both graph algebras and homeomorphism $C^*$-algebras. IV. Pure infiniteness}, J. Funct. Anal. {\bf 254} (2008), 1161--1187.

\bibitem{KP2000} A. Kumjian and D. Pask, \emph{Higher rank graph $C^*$-algebras}, New York J. Math. {\bf 6} (2000), 1--20.

\bibitem{Kun} M. Kunze, \emph{Zappa products}, Acta Math. Hung. {\bf 41} (1983), 225--239.

\bibitem{lr} M. Laca and I. Raeburn, \emph{Phase transition on the Toeplitz algebra of the affine semigroup over the natural numbers}, Adv. Math. {\bf 225} (2010), 643--688.

\bibitem{lrrw} M. Laca, I. Raeburn, J. Ramagge, and M.F. Whittaker  \emph{Equilibrium states on the Cuntz-Pimsner algebras of self-similar actions}, http://arxiv.org/pdf/1301.4722.pdf

\bibitem{LL} N.S. Larsen and X. Li, \emph{The 2-adic ring $C^*$-algebra of the integers and its representations}, J. Funct. Anal. {\bf 262} (2012), 1392--1426.

\bibitem{Law} M.V. Lawson  \emph{A correspondence between a class of semigroups and self-similar group actions. I}, Semigroup Forum {\bf 76} (2008), 489--517.

\bibitem{Li2012} X. Li, \emph{Semigroup $C^*$-algebras and amenability of semigroups},   J. Funct. Anal. {\bf 262} (2012), 4302--4340.

\bibitem{LS} R.C. Lyndon and P.E. Schupp, {\em Combinatorial group theory}, Ergebnisse der Mathematik und ihrer Grenzgebiette {vol. 89}, Springer-Verlag, Berlin, 1977.

\bibitem{m} G.J. Murphy, {\em Ordered groups and Toeplitz algebras}, J. Operator Theory {\bf 18} (1987), 3033--326.

\bibitem{nek_jot} V. Nekrashevych, {\em Cuntz-Pimsner algebras of group actions}, J. Operator Theory {\bf 52} (2004), 223--249.

\bibitem{nek_book} V. Nekrashevych, {\em Self-similar groups}, Math. Surveys and Monographs {vol. 117}, Amer. Math. Soc., Providence, 2005.

\bibitem{N2009} V. Nekrashevych, 
\emph{$C^*$-algebras and self-similar groups}, 
J. reine angew. Math. \textbf{630} (2009), 59--123.

\bibitem{n} A. Nica, {\em $C^*$-algebras generated by isometries and Wiener-Hopf operators}, J. Operator Theory {\bf 27} (1992), 17--52.

\bibitem{RS2007} D.I. Robertson and A. Sims, \emph{Simplicity of $C^*$-algebras associated to higher-rank graphs}, Bull. Lond. Math. Soc. {\bf 39} (2007), 337--344.

\bibitem{RaeburnSims:JOT05} I. Raeburn  and A. Sims, \emph{Product systems of graphs and the {T}oeplitz algebras of higher-rank graphs}, J. Operator Theory \textbf{53} (2005), 399--429.

\bibitem{sy} A. Sims and T. Yeend, {\em $C^*$-algebras associated to product systems of Hilbert bimodules}, J. Operator Theory \textbf{64} (2010), 349--376.

\bibitem{Sp} J. Spielberg, \emph{$C^*$-algebras for categories of paths associated to the Baumslag-Solitar groups},  J. Lond. Math. Soc. {\bf 86} (2012), 728--754.

\bibitem{Szep1} J. Sz\'{e}p, {\em On factorisable, not simple groups}, Acta Univ. Szeged. Sect. Sci. Math. {\bf 13} (1950), 239--241.

\bibitem{Szep2} J. Sz\'{e}p, {\em \"Uber eine neue Erweiterung von Ringen. I}, Acta Sci. Math. Szeged {\bf 19} (1958), 51--62.

\bibitem{Szep3} J. Sz\'{e}p, {\em Sulle strutture fattorizzabili}, Atti Accad. Naz. Lincei Rend. Cl. Sci. Fis. Mat. Nat. (8) {\bf 32} (1962), 649--652.

\bibitem{Z} G. Zappa, {\em Sulla costruzione dei gruppi prodotto di due dati sottogruppi permutabili tra loro}, Atti Secondo Congresso Un. Mat. Ital., Bologna, 1940. Edizioni Rome: Cremonense (1942), 119--125.



\end{thebibliography}
\end{document}